\documentclass{amsart}
\usepackage{setspace}
\usepackage{mathrsfs}
\newcommand{\salg}{\mathcal{S}}
\newcommand{\rcfield}{\vfield}
\newcommand{\vfield}{\mathscr{K}}
\newcommand{\coloneqq}{:=}

\newcommand{\doag}{\R}
\newcommand{\vring}{\mathscr{O}}

\newcommand{\xsec}{\operatorname{cs}}

\newcommand{\slin}{S}
\usepackage{frenchineq}
\usepackage{amsmath}
\usepackage{amsthm}
\usepackage{amssymb}
\usepackage{verbatim}
\usepackage{subfigure}
\usepackage{verbatim}
\usepackage{url}
\numberwithin{equation}{section}
\usepackage{textcomp}
\usepackage{xy}
\usepackage{tikz}
\usepackage{color}

\usepackage{epstopdf}

\usepackage{amsmath}
\usepackage{amsthm}
\usepackage{amssymb}
\usepackage{verbatim}
\usepackage{subfigure}
\usepackage{verbatim}
\usepackage{url}
\numberwithin{equation}{section}
\usepackage{textcomp}
\usepackage{stmaryrd}
\usepackage{xy}
\usepackage{color}

\usepackage{epstopdf}
\usepackage{mathrsfs}
\usepackage{hyperref}
\usepackage{cleveref}
\usepackage{tikz}
\usepackage{color}

\usepackage{bbm}    

\usepackage{bm}

\usepackage[colorinlistoftodos,bordercolor=orange,backgroundcolor=orange!20,linecolor=orange,textsize=scriptsize]{todonotes}

\definecolor{Mygrey}{gray}{0.75}

\def\displayandname#1{\rlap{$\displaystyle\csname #1\endcsname$}%
                      \qquad \texttt{\char92 #1}}

\newtheorem*{lemma*}{Lemma}

\makeatletter
\def\smallunderbrace#1{\mathop{\vtop{\m@th\ialign{##\crcr
   $\hfil\displaystyle{#1}\hfil$\crcr
   \noalign{\kern3\p@\nointerlineskip}%
   \tiny\upbracefill\crcr\noalign{\kern3\p@}}}}\limits}
\makeatother

\newcommand{\puiseuxalpha}{{\bm \alpha}}
\newcommand{\puiseuxlambda}{{\bm \lambda}}
\newcommand{\puiseuxrho}{{\bm \rho}}
\newcommand{\puiseuxA}{{\bm A}}
\newcommand{\puiseuxI}{{\bm I}}
\newcommand{\puiseuxJ}{{\bm J}}
\newcommand{\puiseuxB}{{\bm B}}

\newcommand{\puiseuxF}{{\bm F}}
\newcommand{\puiseuxD}{{\bm D}}
\newcommand{\puiseuxM}{{\bm M}}
\newcommand{\puiseuxP}{{\bm P}}
\newcommand{\puiseuxf}{{\bm f}}
\newcommand{\puiseuxg}{{\bm g}}

\newcommand{\plucker}{\Delta}

\newcommand{\stiefel}{\phi}

\definecolor{Mygrey}{gray}{0.75}

\def\displayandname#1{\rlap{$\displaystyle\csname #1\endcsname$}%
                      \qquad \texttt{\char92 #1}}

\makeatletter
\def\url@leostyle{%
  \@ifundefined{selectfont}{\def\UrlFont{\sf}}{\def\UrlFont{\small\ttfamily}}}
\makeatother
\DeclareMathAlphabet{\mathbbold}{U}{bbold}{m}{n}
\newcommand{\zero}{\mathbbold{0}}
\newcommand{\unit}{\mathbbold{1}}

\newcommand{\minusinfty}{-\infty}

\newcommand{\per}{\operatorname{per}}
\newcommand{\Id}{\operatorname{Id}}

\newcommand\sign{{\operatorname{sgn}}}
\newcommand\Gr{{\operatorname{Gr}}}

\newcommand{\rmax}{\mathbb{R}_{\max}}
\newcommand{\R}{\mathbb{R}}
\newcommand{\K}{\mathbb{K}}
\newcommand{\Kpos}{\K_{>0}}
\newcommand{\Knneg}{\K_{\geq 0}}

\newcommand{\pat}{\operatorname{pat}}
\newcommand\val{{\operatorname{val}}}

\newcommand\TP{{\mathsf{TP}}}  %
\newcommand\TN{{\mathsf{TN}}}
\newcommand\trop{{\operatorname{trop}}}
\newcommand\Trop{{\operatorname{Trop}}}

\newcommand\DD{{\mathsf{DD}^{\trop}}}
\newcommand\NDD{{\mathsf{SDD}^{\trop}}}
\newcommand\GL{{\operatorname{GL}}}
\newcommand\diag{{\operatorname{diag}}}

\newtheorem{innercustomthm}{Theorem}
\newenvironment{customthm}[1]
  {\renewcommand\theinnercustomthm{#1}\innercustomthm}
  {\endinnercustomthm}

\newcommand{\mgroup}[1]{#1^*}
\newtheorem{thm}{Theorem}[section]
\newtheorem{theorem}[thm]{Theorem}
\newtheorem{pro}[thm]{Proposition}
\newtheorem{lem}[thm]{Lemma}

\newtheorem{cor}[thm]{Corollary}
\newtheorem{corollary}[thm]{Corollary}

\theoremstyle{definition}
\newtheorem{df}[thm]{Definition}
\newtheorem{definition}[thm]{Definition}

\theoremstyle{remark}
\newtheorem{rem}[thm]{Remark}
\newtheorem{exa}[thm]{Example}
\newtheorem{example}[thm]{Example}

\title{Tropical totally positive matrices}

\author{{S}t\'ephane Gaubert}
\address{St\'ephane Gaubert,
INRIA Saclay--\^Ile-de-France and CMAP, \'Ecole 
polytechnique, CNRS. Address: CMAP, \'Ecole polytechnique,
Route de Saclay,
91128 Palaiseau Cedex, France.}
\email{Stephane.Gaubert@inria.fr}

\author{{A}di Niv}
\address{Adi Niv,
Mathematics Department, Science Faculty, Kibbutzim College.
Address: Kibbutzim College, 149 Namir Rd., Tel-Aviv, Israel.}
\email{adi.niv@smkb.ac.il}

\thanks{The first author has been partially supported by the 
Gaspard Monge Program (PGMO) of FMJH and EDF, by a public grant as part of the
Investissement d'avenir project, reference ANR-11-LABX-0056-LMH,
LabEx LMH, and by the MALTHY Project of the ANR Program. The second author was sported by the 
French Chateaubriand grant and INRIA postdoctoral fellowship.} 
\thanks{We thank Gleb Koshevoy for suggesting to investigate the present topic and for discussions. We also   
thank Charles Johnson for pointing out the interpretation of our results in terms of Monge matrices, and Benjamin Schr\"oter for helpful comments.}

\begin{document}

\begin{abstract}
We investigate the tropical analogues of totally positive and totally nonnegative matrices. 
These arise when considering the images by the nonarchimedean valuation of the 
corresponding classes of matrices over a real nonarchimedean valued field, like the 
field of real Puiseux series. We show that the nonarchimedean
valuation sends the totally positive matrices precisely to the Monge matrices. 
This leads to explicit polyhedral representations of the tropical
analogues of totally positive and totally nonnegative matrices. 
We also show that tropical totally nonnegative matrices with a finite permanent
can be factorized  in terms of elementary
matrices. We finally determine the eigenvalues of tropical
totally nonnegative matrices, and relate them with the eigenvalues
of totally nonnegative matrices over nonarchimedean fields.

\vskip 0.15 truecm

\noindent \textit{Keywords: Total positivity; total nonnegativity; tropical geometry;   compound matrix;  permanent; Monge matrices; Grassmannian; Pl\"ucker coordinates.}
\vskip 0.1 truecm

\noindent \textit{AMSC: 15A15 (Primary), 15A09, 15A18, 15A24, 15A29, 15A75, 15A80, 
15B99.} 	
\end{abstract}

\maketitle



\section{Introduction}
\subsection{Motivation and background}
A real matrix is said to be {\em totally positive} (resp.~{\em totally nonnegative}) 
if all its minors are positive (resp.~nonnegative). These matrices
arise in several classical fields, such as oscillatory matrices (see e.g.~\cite[\S4]{TPM}), 
or approximation theory (see e.g.~\cite{kluwertotalpositivity});
they have appeared more recently 
in the theory of canonical bases for quantum groups~\cite{berensteinparametrization}.
We refer the reader to the monograph of Fallat and Johnson in~\cite{Fallat&Johnson} or to the survey of Fomin and Zelevinsky~\cite{F&Z} for more information.
Totally positive/nonnegative matrices can be defined over any real
closed field, and in particular, over nonarchimedean fields, like
the field of Puiseux series with real coefficients. 
In this paper, we characterize
the set of images of such matrices by the nonarchimedean valuation
which associates to a Puiseux series its leading exponent.
To do so, we study the tropical analogues of totally positive
and totally nonnegative matrices.

To describe further our results, it is convenient to recall some basic notions. 
The max-plus 
(or tropical) semifield, denoted by~$\rmax$,
is the set~$\mathbb{R}\cup\{-\infty\}$ equipped with the laws~$a\oplus b:= \max(a,b)$ 
and~$a\odot b:= a+b$. 
(See for instance~\cite{BCOQ92,TAG,MPA,butkovicbook,MacStur}.) It has a zero element,~$\zero=-\infty$, and a unit 
element,~$\unit=0$. We abuse notation by using the same symbol, $\rmax$
for the semifield and for its ground set. 
The tropical numbers can be thought of as the images
by the valuation of the elements of a nonarchimedean field.
A convenient choice of field, denoted by $\K$, consists
of 
(generalized) Puiseux series with real coefficients and real exponents.
Such a series can be written as 
\begin{align}
\puiseuxf:= \sum_{k\geq 0} a_k t^{b_k}\enspace , 
\label{e-puiseux}
\end{align}
where~$a_k\in \R$, $b_k\in \R$, and~$(b_k)$ is a decreasing sequence
converging to~$-\infty$. 
The {valuation} of~$\puiseuxf$ is defined to be the largest exponent of~$\puiseuxf$, 
i.e., ~$\val (\puiseuxf):= \sup \{b_k \mid a_k\ne 0\}$,
with the convention that~$\val (0)=\minusinfty$. 

The choice of this specific field of formal series is only
to keep the exposition concrete.
What matters is that the nonarchimedean field is real
closed, that its value group is $\R$, and that its valuation
has a certain property called convexity; we refer the reader
to \Cref{subsec-nonarch} for more details on the setting in which
our results hold. We emphasize in particular that our main results also
apply to fields of absolutely convergent series
with real exponents like the ones of~\cite{Dries1998} and more generally to Hardy fields of polynomially bounded o-minimal structures~\cite{alessandrini2013}.

A nonzero series is said to be {\em positive} if its leading coefficient  is positive. We denote 
by~$\Kpos$
the set of positive series, and we denote by~$\Knneg:= \Kpos\cup \{0\}$ the set of 
{\em nonnegative} series. The map~$\val$ satisfies, for all $\puiseuxf,\puiseuxg\in \Knneg$, 
\[
\val(\puiseuxf+\puiseuxg) = \max(\val(\puiseuxf),\val(\puiseuxg))\enspace ,\qquad \val(\puiseuxf\puiseuxg) = \val(\puiseuxf)+\val(\puiseuxg)
\enspace .
\] 
In tropical algebra, we are interested in properties of objects defined over~$\mathbb{K}$ 
that can be inferred from this valuation.
We  study here the images by the valuation of the 
classical classes of totally positive or totally nonnegative matrices
over~$\K$. To do so, we associate to a~$d\times d$
matrix~$\puiseuxA=(\puiseuxA_{i,j})$ with entries
in~$\Knneg$, the matrix~$A$ with entries~$A_{ij}:= \val (\puiseuxA_{i,j})$ in~$\rmax$. 
We say that~$\puiseuxA_{i,j}$ is a \textit{lift} of~$A_{i,j}$, and ~$\puiseuxA$ is a \textit{lift} of~$A$, 
with the convention that~$0$ is the lift of~$\minusinfty$.
The {\em tropical permanent} of~$A$ is defined as 
\begin{align}
\per (A):= 
\max_{\sigma \in S_d} \sum_{i\in[d]} A_{i,\sigma(i)} \enspace ,
\label{e-def-tper}
\end{align}
where~$S_d$ is the set of permutations on~$[d]:=\{1,\dots,d\}$, and~$\sum_{i\in[d]} 
A_{i,\sigma(i)}$ is the \textit{weight} of the permutation~$\sigma$ in~$\per(A)$.

We say that the matrix~$A$ is (tropically) {\em sign-nonsingular}
if~$\per (A) \neq \minusinfty$ and if all the permutations~$\sigma$ of maximum weight 
have the same parity. Otherwise,~$A$ is said to be (tropically) {\em sign-singular}. 
We refer the reader to~\cite{shader} for more background
on the classical notion of sign-nonsingularity, and to~\cite{NSM,benchimol2013} for its tropical version. When~$A$ is  sign-nonsingular,
it is easily seen that
\[
\val (\det(\puiseuxA)) = \per(A)  \enspace ,
\] 
and the sign of~$\det(\puiseuxA)$ coincide with the sign of every permutation of 
maximal weight in~$\per(A)$.
A \textit{tropical minor}   is defined as the tropical permanent of a square 
submatrix. A tropical minor 
is said to be  \textit{tropically positive} (resp.~\textit{tropically negative}) if all its 
permutations of maximum weight   are even (resp.~odd). It is said to be 
\textit{tropically nonnegative} (resp.~\textit{tropically nonpositive}) 
if either the above condition
holds or the submatrix is  sign-singular. This terminology can be justified
by embedding the max-plus semiring in the symmetrized max-plus semiring~\cite{LS,AGG14}.

 A tropical matrix is said
to be {\em tropical totally positive} if  its entries are in~$\mathbb{R}$, and    
all its  minors are  tropically positive. It is
said to be {\em tropical totally nonnegative}  if  its entries are in~$\rmax$, and    
all its  minors are    tropically nonnegative.


\subsection*{Summary of notation} 
It is convenient now to list the main notations used in the manuscript.
We follow the notation used by Fallat and Johnson~\cite{Fallat&Johnson}
for various classes of totally nonnegative matrices. 

We denote by~$\TP$ (resp.~$\TN$) the set of totally positive (resp.~totally 
nonnegative) matrices over a field. 
We write~$\TP(\mathbb{K})$ or $\TP(\mathbb{R})$, indicating
the ground field in parenthesis, when necessary,
and we use a similar notation for $\TN$. 
The set~${\TP_t}$ 
(resp.~${\TN_t}$) denotes matrices whose  minors of size at most~$t$ 
are   positive (resp.~nonnegative).
Similarly, we shall denote by~$\TP^{\trop}$ (resp.~$\TN^{\trop}$) the set of tropical totally positive (resp.~tropical totally nonnegative) matrices,
which have entries in~$\rmax$. 
In general the entries of a matrix
in $\TN^\trop$ may take the $-\infty$ value.
We denote  by~$\TN^\trop(\R)$  the subset of matrices in $\TN^\trop$ 
whose entries belong to $\R$.
We also denote by~${\TP^{\trop}_t}$ (resp.~${\TN^{\trop}_t}$) 
the set of matrices with entries in $\rmax$ whose every 
tropical minor of size at most~$t$ is tropically 
positive (resp.\ tropically nonnegative).
Moreover, ${\TN^{\trop}_t}(\R)$ denotes the subset
of $\TN^{\trop}_t$ consisting of matrices with entries
in $\R$.

We also denote by~$\DD$ the set of matrices such that every 
submatrix is tropical diagonally dominant. In this setting,
a matrix $A$ is said to be tropical diagonally 
dominant if $\per(A)=\sum_{i\in[n]} A_{i,i}$.
If $\per(A) >\sum_{i\in [n]}A_{i,\sigma(i)}$ for all permutations
$\sigma$ distinct from the identity,
the matrix $A$ is said to be tropical {\em strictly} diagonally
dominant; the set of such matrices is denoted by~$\NDD$.

We draw the reader's attention to the notation $\TP$,
used in~\cite{Fallat&Johnson} to denote totally positive matrices.
This should not be confused with the notation $\mathbb{T}\mathbb{P}^{n-1}$, 
used in~\cite{DS} to denote the stratum of the $(n-1)$-dimensional \textit{tropical projective space} that consists of rays generated by finite vectors;
here we use the notation $\mathbb{P}^{n-1}(\rmax)$ for the tropical projective
space.

\subsection{Main results}
Our main results relate the classical and tropical notions of total nonnegativity.  The following theorem follows by combining \Cref{trop2}, \Cref{coro-fin} and \Cref{diagcor} below, 
it shows that the image by the valuation of the set of tropical total positive
matrices is determined by the tropical positivity of $2\times 2$ minors.
\begin{customthm}{A}\label{thA}
We have
\[
\TN_2^\trop(\R)
 =\TN^\trop (\R)
= \val(\TN(\K^*)) =\val(\TP) \enspace.
\]
\end{customthm}
We shall see
that~$\TN^\trop_2(\R)$  is precisely the 
set of \textit{Monge matrices}, named after Gaspard Monge,
as they arise in optimal  
transportation problems.
The set of Monge matrices has an explicit polyhedral parametrization
which follows from results of~\cite{BKR} and~\cite{FIEDLER},
see~\Cref{prop-new}. 

Another main result provides a tropical analogue
of the Loewner--Whitney theorem~\cite{Whitney,Loewner}
and~\cite[Theorem~12]{F&Z}.
The classical theorem shows that any invertible totally nonnegative
matrix is a product of nonnegative elementary Jacobi matrices.
\begin{customthm}{B}
We have
\[
\val(\GL_n\cap\TN)=
\langle\text{tropical~Jacobi~elementary~matrices}\rangle\enspace .
\]
\end{customthm}
This is part of \Cref{rltn} below. 
Here, $\GL_n$ denotes the set of invertible~$n\times n$ matrices over~$\K$,
and $\langle\cdot\rangle$ denotes the multiplicative semigroup generated
by a set of matrices (i.e., the set of finite products of matrices in this set).
Tropical Jacobi elementary matrices are defined in a way analogous to the
classical situation, see~\Cref{factlift}.
The semigroup they generate was studied in~\cite{FTM}: whereas classically,
every nonsingular matrix can be factored in terms of 
elementary matrices, the same is not true in the tropical
setting. Nevertheless, it was shown there
that the set of $3\times 3$ tropical matrices which
admits such a factorization admits a combinatorial characterization.
The present result shows that this characterization
can be interpreted in terms of total nonnegativity,
and provides some generalization to the $n\times n$ case.

A complete comparison of the classes of matrices studied
in the present paper can be found in \Cref{thC} and
\Cref{Table1,Table2} below.

The paper comprises several other results. 
In particular, \Cref{TNc}, which holds for matrices
over Puiseux series, and in particular, for matrices
over the field of real numbers, is a linear algebra
result which may be of independent interest, it shows that 
if the $2\times 2$ minors of a matrix with positive entries
are positive and ``sufficiently away'' from $0$, then, this
matrix is totally positive. 
Propositions~\ref{cc&ee} and~\ref{eede} characterize the valuations
of the eigenvalues of a totally positive matrix with entries in $\K$, showing
these are nothing but the valuations of the diagonal entries of the matrix.
Corollary~\ref{PNTP} provides a tropical analogue of the representation
of totally nonnegative matrices as weight matrices of planar networks.

\subsection{Related results and other approaches}
\label{subsec-related}
As mentioned above, a first source of inspiration of the present
work is the classical theory of totally positive matrices;
Fallat and Johnson~\cite{Fallat&Johnson}
and Fomin and Zelevinsky~\cite{F&Z} gave recent accounts of this theory.
We exploited a characterization
of Monge matrices, obtained by Burkard, Klinz, Rudolf and
Fiedler~\cite{BKR,FIEDLER}. 

The notion of ``positivity'' has other incarnations in linear
algebra. In particular, the tropical analogues of {\em positive definite}
matrices have been studied by Yu, who showed in~\cite{YU} that the image
by the valuation
of the set of symmetric positive definite matrices 
over the field of Puiseux series
is characterized by the nonnegativity of its
principal~$2\times 2$ minors. More generally, the tropicalization
of ``generic'' spectrahedra involves $2\times 2$ minors~\cite{issac2016}.
We show
that a somehow analogous property,~$\val(\TP)=\TN_2^\trop(\R)$, is valid for
totally positive matrices.

Another approach to total positivity
arises by considering,
following Postnikov~\cite{POST},
the totally nonnegative (or positive) Grassmannian.
The latter  
consists of the elements of the Grassmannian that have nonnegative 
(or positive) Pl\"ucker coordinates. 
A survey on totally nonnegative Grassmannian and a new approach to  
Grassmann polytopes via canonical bases was given by Lam in~\cite{LAM}. 
Speyer and Williams studied in~\cite{SW05}
the ``tropical totally positive Grassmannian''.
The latter arises
as the image by the valuation of the totally
positive Grassmannian over the nonarchimedean field of Puiseux series
with real coefficients. Whereas the classical Grassmanian can be realized
as the image of the map which sends a full rank matrix to its (projective)
Pl\"ucker coordinates, the same approach, transposed to the tropical setting, 
only yields an inner approximation of the tropical Grassmannian,
as shown by Herrmann, Joswig and Speyer~\cite{spst} and
Fink and Rinc\'on~\cite{rinconfink}.
Similarly, the totally positive Grassmannian can be realized
as an image of the set of totally positive matrices. We shall see in \Cref{subsec-other} that when transposed to the tropical setting, 
this approach leads only to a linear parametrization of a subset
of the tropical totally positive Grassmannian
by tropical totally nonnegative matrices. 



\section{Preliminaries: nonarchimedean amoebas of semialgebraic sets}
\label{subsec-nonarch}

It is convenient to summarize here the main properties
of valued fields which will be used. We refer the reader to~\cite{engler_prestel_valued_fields}
for background on valued fields. 

We consider a field $\vfield$ equipped with a total order relation $\geq$.
We denote by $\vfield_{\geq 0}:=\{f\in\vfield\mid f\geq 0\}$ the set of
nonnegative elements of $\vfield$ and by 
$\vfield_{> 0}:=\{f\in\vfield\mid f> 0\}$ the set of its positive elements.
We also assume
that $\vfield$ is equipped with a non-archimedean valuation, i.e.,
a map $\val: \vfield\to \R\cup\{-\infty\}$ satisfying 
the following properties
\begin{align}\label{e-init}
\val (f+g) \leq \max(\val (f), \val (g)), \quad
\val (fg) = \val (f)  + \val (g) , \quad 
\val (f) = -\infty \iff f =0 \enspace. 
\end{align}
The image of $\vfield^*:=\vfield\setminus\{0\}$ by the non-archimedean valuation
is a subgroup of $(\R,+)$, called the {\em value group}.

The field $\vfield$ possesses a subring, $\vring \coloneqq \{f \in \vfield \colon \val(f) \leq 0 \}$. 
We say that the valuation $\val$ is \emph{convex}
if it satisfies the following property: for every $f \in \vring$ and every $g
 \in \vfield$ we have the implication
\begin{align}\label{e-def-convex}
0 \leq g \leq f \implies g \in \vring \, .
\end{align}
This is equivalent to the following property: 
for all $f,g\in\vfield$, 
\begin{align}
f,g\geq 0 \implies \val (f+g) = \max(\val (f), \val (g))
\enspace .
\label{e-morphism}
\end{align}




An ordered field is said to be {\em real closed} if the set
of nonnegative elements is precisely the set of squares,
and if every polynomial of odd degree with coefficients
in this field has at least one root in the same field.
A theorem of Tarski shows that the first order theory of real closed fields is complete, see~\cite{marker}, Coro.~3.3.16.
This entails that a property expressed by a first order
sentence in the language of ordered fields which is valid in a special real closed
field, like $\R$, is valid in any real closed field.  We shall make use of this property in the sequel.

If $\rcfield$ is a real closed field, then we say that a subset $\mathcal{S} \subset \rcfield^{n}$ is \emph{basic semialgebraic} if it is of the form
\[
\mathcal{S} = \{(f_1,\dots,f_n) \in \rcfield^{n} \colon \forall i\in[p],\; P_{i}(f_1,\dots,f_n) > 0  \land \forall i \in [q]\setminus[p]
, P_{i}(f_1,\dots,f_n) = 0 \} \, ,
\]
where $P_{1},\dots,P_{q}$ are multivariate polynomials with coefficients in $\rcfield$, and $1\leq q\leq p$.
We say that $\salg$ is \emph{semialgebraic} if it is a finite union of basic semialgebraic sets. 

Gelfand, Kapranov, and Zelevinsky introduced in~\cite{GelfandKapranovZelevinsky} the notion of {\em amoeba}
of an algebraic variety $V$ included in $(\mathbb{C}^*)^n$, as the image of this
variety by the map $(z_i) \mapsto (\log |z_i|)$. Amoebas have
also been considered in the nonarchimedean setting; $\mathbb{C}$ 
is now replaced by a non-archimedean field, and the log-of-modulus
map is replaced by the nonarchimedean valuation~\cite{kapranov}. In the present
work, we will be interested by amoebas of subsets
defined by inequalities as well as by equalities. This leads to the following notion.
\begin{definition}
If $\mathcal{S}$ is a semi-algebraic subset of $\rcfield_{>0}^n$, where
$\rcfield$ is a real closed field equipped with a nonarchimedean
valuation $\val$, the {\em amoeba} of $\mathcal{S}$ is the set
$\val (\mathcal{S}):=\{(\val f_1,\dots,\val f_n)\mid (f_1,\dots,f_n)\in \mathcal{S}\}\subset \R^n$.
\end{definition}

Such amoebas have a polyhedral structure. Recall that a set $\slin \subset \doag^{n}$ is \emph{basic semilinear} if it is of the form
\[
\slin = \{(x_1,\dots,x_n) \in \doag^{n} \colon \forall i  \in[p],\; \ell_{i}(x_1,\dots,x_n) > h^{(i)} \wedge \forall i \in[q]\setminus[p], \ell_{i}(x_1,\dots,x_n) = h^{(i)} \} \, ,
\]
where $\ell_{1},\dots,\ell_q$ are linear forms with integer coefficients,
$h^{(1)},\dots,h^{(q)} \in \doag$, and $1\leq p\leq q$. We say that $\slin$ is \emph{semilinear} if it is a finite union of basic semilinear sets.

The following result is derived in~\cite{skomra} as a corollary of 
a quantifier elimination result for valued fields of 
Denef~\cite{denef_p-adic_semialgebraic} and Pas~\cite{pas_cell_decomposition}.
A related result was proved by Alessandrini, in the setting of o-minimal geometry~\cite{alessandrini2013}.
\begin{theorem}[{\cite[Theorem~4]{skomra}, see also~\cite[Theorem 3.11]{alessandrini2013}}]\label[theorem]{theorem:image_finite_union_of_ri_polyh}
Let $\vfield$ be a real closed field equipped with a convex nonarchimedean valuation~$\val$ with value group $\R$. Furthermore, suppose that the set $\mathcal{S} \subset \vfield_{>0}^{n}$ is semialgebraic. Then $\val(\mathcal{S})$ is a semilinear
subset of $\R^n$. 
\end{theorem}
It is also known that this 
semilinear subset is closed in the Euclidean topology~\cite[Theorem~10]{skomra}. 

In the sequel, we shall be interested in the amoebas of 
semialgebraic sets over a nonarchimedean field,
especially, the set of totally positive matrices
and the ``totally positive
part'' of the Grassmanian. Hence, it may help to keep in mind \Cref{theorem:image_finite_union_of_ri_polyh}. However, we emphasize that our main
results do not rely on this theorem, but rather proceed by direct characterizations.

We now give examples of real closed nonarchimedean ordered
fields with a convex valuation.
\begin{example}\label{ex-hahn}
A {\em Hahn series} is of the form
\begin{align}
\sum_{\lambda \in \Lambda}
a_\lambda t^\lambda
\label{e-hahn}
\end{align}
where $a_\lambda\in \R\setminus 0$ and 
$-\Lambda$  is a well ordered subset of $\R$,
with the convention that the latter sum is zero if $\Lambda$
is empty. This field is denoted by $[[\R^{(\R,\leq)}]]$ in~\cite[(6.10)]{ribenboim}, where it is shown to be real closed. 
\end{example}

The following smaller field
is a popular choice in tropical geometry~\cite{MARK}.
\begin{example}\label{ex-genpuiseux}
A {\em generalized Puiseux series} is a series
of the form~\eqref{e-hahn}, where $\Lambda$ is either 
empty, or finite, or coincides with the set of elements of a sequence
of real numbers decreasing to $-\infty$. This is precisely
the field considered in the introduction,
and denoted by $\K$ there.
It follows for instance from a result of~\cite{MARK}
that this field is real closed.
Indeed, the latter reference considers the field of formal generalized
Puiseux series with {\em complex} coefficients. This field,
which can be identified to the quadratic extension $\K[\sqrt{-1}]$, is shown
to be algebraically closed in~\cite{MARK}. This implies
that $\K$ is real closed. In the sequel, we shall 
often think of $\K$ as the subset of ``real'' elements of $\K[\sqrt{-1}]$,
extending the classical terminology and notation for complex numbers,
like ``real part'', ``imaginary part'',  and ``modulus'',
 to $\K[\sqrt{-1}]$. E.g., the modulus of $c=a+(\sqrt{-1})b$
with $a,b\in \K$ is $|c|=\sqrt{a^2+b^2}\in \K$.
\end{example}
\begin{example}\label{ex-cvg}
The field $[[\R^{(\R,\leq)}]]$ of Hahn series and the field $\K$ of
generalized Puiseux series have subfields,
consisting of series that are absolutely convergent
for all sufficiently small positive values of $t$. We shall denote by 
$[[\R^{(\R,\leq)}]]_{\text{cvg}}$ and $\K_{\text{cvg}}$ these two fields, respectively.
van den Dries and Speissegger showed
in~\cite[Corollary 9.2]{Dries1998} that 
$[[\R^{(\R,\leq)}]]_{\text{cvg}}$ 
is real closed; indeed, $[[\R^{(\R,\leq)}]]_{\text{cvg}}$ 
is precisely the field of germs of functions definable in a certain 
o-minimal structure $\R_{\text{an}*}$.
They also observed in Section 10.2, {\em ibid.},
that the same proofs apply to $\K_{\text{cvg}}$, 
which entails in particular that this field is also real closed.
We finally note that the field $\K_{\text{cvg}}$
is isomorphic to the field of 
(absolutely convergent)
{\em generalized Dirichlet series} considered  by Hardy and Riesz~\cite{hardy}.
Ordinary Dirichlet series are of the form $\sum_{n\geq 1} a_n n^{-s}$; 
they can be identified to series 
of the form~\eqref{e-puiseux}, setting $\lambda_n = -\log n$, with the substitution $s:=\exp(t)$. 
\end{example}



A map $\xsec \colon \R \to \vfield^{*}$ is called a \emph{cross-section} 
if it is a multiplicative morphism such that $\val \circ \xsec$ is the identity map. For instance, if $\vfield$ is any of the ordered fields in Examples~\ref{ex-hahn}--\ref{ex-cvg}, the map $y\mapsto t^y$ is a cross section. We shall frequently use this cross section in the proofs which follow.
More generally, every real closed valued field with a convex
nonarchimedean valuation and value group $\R$
has a cross-section~\cite[Lemma~3]{skomra}.

To keep our exposition as concrete as possible, we shall assume in the sequel
that $\vfield=\K$ is the field of generalized Puiseux
series considered in Example~\ref{ex-genpuiseux}. Our results hold, without changes, if $\K$ is replaced
by the field of Hahn series (Example~\ref{ex-hahn}), or by fields
of absolutely convergent series like the ones of Example~\ref{ex-cvg}.
More generally, we leave it
to the reader to check that our arguments hold
in any real closed field with a convex nonarchimedean valuation
and value group $\R$,
referring to~\cite{alessandrini2013} and to~\cite{skomra} for more details on
the tropicalization of semialgebraic sets over ordered nonarchimedean fields,
including (in~\cite{skomra}) situations in which the value group is not necessarily $\R$.

\section{Characterizing tropical total positivity in terms of $2\times 2$ minors}\label{sec-global}

\label{2dd} 

This section follows the notation
and results in~\cite[\S 3.1]{Fallat&Johnson}. 
Total positivity of an~$n\times m$ matrix may be verified by  the 
positivity of its~$nm$ initial minors~\cite{Mariano&Pena}. 
We  show that in the tropical
setting, total nonnegativity can be checked by considering
$nm$ 
~$2\times 2$ solid minors. We also show
that if these minors are non-$\zero$, then the matrix 
is uniquely determined by them. 
We shall also see that tropical total positivity implies
some kind of diagonal dominance. 


We shall need the following immediate fact.
\begin{lem}\label{inv} Let $\pi\in S_n$ be a permutation 
different from the identity permutation~$\Id$.
 There exists $i,j$ such that $i>j$ and $\pi(i)<\pi(j)$, called an inversion of~$i,j$ in~$\pi$. 
 Moreover, we may choose $j=\pi(i)$.\hfill\qed 
\end{lem} 


We recall that a real matrix~$A$ is a \textit{Monge matrix} if and only 
if it satisfies the \textit{Monge property}
\begin{align}
A_{i,j}+ A_{i',j'}\geq A_{i,j'}+ A_{i',j},
\qquad \text{for all } i<i' \text{ and } j<j' \enspace .
\label{e-def-Monge}
\end{align}
A \textit{strict} Monge matrix is obtained by requiring the above
inequalities to be strict. (More precisely, the Monge property is defined by the reversed inequalities, 
whereas the definition above is the anti-Monge property. We omit the notion ``anti" throughout.) 
Note that~$\TP^{\trop}_2$ and~$\TN^{\trop}_2(\R)$ are by definition the sets of 
strict Monge matrices and respectively  Monge matrices. 
The following observation is classical.
\begin{lem}[\cite{BKR}]\label{subsequent0}
A real matrix is a (strict) Monge matrix 
if and only if the (strict) 
relations~\eqref{e-def-Monge} hold for consecutive values
of $i,i'$ and of $j,j'$.\hfill \qed
\end{lem}

We shall also use the following easy observation,
showing that the valuation restricted to the set
of nonnegative (generalized) Puiseux series
is an order preserving map:
\begin{align}
f,g\in \K_{\geq 0} \text{ and }
f \geq g \implies \val f \geq \val g \enspace.
\label{val-op}
\end{align}


We start by an elementary lemma.
\begin{lem}\label{2trop2}  $\val(\TN_2)=\TN^\trop_2$. \end{lem}

\begin{proof} 
Suppose that $\puiseuxA=(\puiseuxA_{ij})\in \TN_2$,
and let $A_{ij}:= \val \puiseuxA_{ij}$.
By definition of $\TN_2$, for all $i<j$ and $k<l$, 
we have $\puiseuxA_{ik}\puiseuxA_{jl} \geq \puiseuxA_{il}\puiseuxA_{jk}$,
and by~\eqref{val-op},
${A}_{ik}+{A}_{jl} \geq {A}_{il}+{A}_{jk}$,
which implies that $A\in \TN^\trop_2$. Conversely, if $A\in \TN^\trop_2$,
we immediately check that the {\em canonical lift} $\puiseuxA$ of $A$,
$\puiseuxA_{ij}:= t^{A_{ij}}\in \K$, belongs to $\TN$.


\end{proof}

We now show that the tropical totally nonnegative real matrices are precisely the Monge matrices. 

\begin{theorem}
\label{trop2} $\ \TP^{\trop}=\TP^{\trop}_2\ $ and $\ \TN^{\trop}=
\TN^{\trop}_2$.
\end{theorem}

In the proof, and in the sequel, for all $I\subset[n]$
and $J\subset[m]$, we denote by $A_{I,J}$ the $I\times J$ submatrix of $A$. 
\begin{proof}
By definition $\TP^{\trop}\subset \TP^{\trop}_2$ and~$\TN^{\trop}
\subset \TN^{\trop}_2$.

Assume~$A\notin \TN^{\trop}$. Therefore, for some~$I,J$ s.t.~$|I|=|J|>2$,  
all the permutations of maximum weight in~$\per(A_{I,J})$ are odd.  Let~$\pi$ be such an odd permutation, and denote by~$M$ the~$2\times 2$ 
submatrix~$A_{\{j,i\},\{\pi(i),\pi(j)\}}$,  for some inversion of~$i>j$ in~$\pi$. 
Suppose that~$\per(M)=A_{j,\pi(i)}\odot A_{i,\pi(j)}\geq A_{i,\pi(i)}\odot A_{j,\pi(j)}$, and consider~$\sigma:=\pi\circ (\pi(i)\ \pi(j))$, where $(k,l)$ 
denotes the transposition of indices $k,l$, and $\circ$
denotes the composition of permutations. Then, $\sigma$
is still of maximum weight
in~$\per(A_{I,J})$, and it is even, contradicting the assumption.
It follows that~$\per(M)=
A_{i,\pi(i)}\odot A_{j,\pi(j)}>A_{j,\pi(i)}\odot A_{i,\pi(j)} ,$ and~$A\notin 
\TN^{\trop}_2$.
 
Similarly, if~$A\notin\TP^{\trop}$, then there exists
at least one odd permutation $\pi$ 
of maximum weight in~$\per(A_{I,J})$. If~$\per(M)=A_{j,\pi(i)}\odot A_{i,\pi(j)}
>A_{i,\pi(i)}\odot A_{j,\pi(j)},$ then the weight of the permutation~$\pi\circ (\pi(i)\ 
\pi(j))$ is strictly bigger than the maximal weight of~$\pi$ in~$\per(A_{I,J})$. 
Thus~$\per(M)=A_{i,\pi(i)}\odot A_{j,\pi(j)}\geq A_{j,\pi(i)}\odot A_{i,\pi(j)}  ,$ 
and therefore~$A\notin \TP^{\trop}_2$.
Additionally, if~$\per(A_{I,J})=\minusinfty$, then~$A$ has~$\minusinfty$ entry, and 
therefore~$A\notin \TP^{\trop}_2$.
\end{proof}

It seems to be a general principle in tropical geometry
that properties over the tropical semiring translate to 
weaker or approximate properties 
over fields, see e.g.~\cite{1307.3681,logmajorization2013} for applications
of this principle to the localization of roots
of polynomials. According to the same principle,
we expect the characterization of tropical totally
positive matrices in Theorem~\ref{trop2} to translate
to a sufficient condition for the total positivity of matrices
over a real field. We next give such a condition.

Given $C\geq 1$, we denote by~$\TN_{2,C}$ the set of matrices~$\puiseuxA$ such that
\[  \puiseuxA_{i,j}\puiseuxA_{i',j'}\geq C\puiseuxA_{i,j'}\puiseuxA_{i',j},\ \forall i<i'\text{ and }j<j'.
\]
The set~$\TP_{2,C}$ denotes the subset of ~$\TN_{2,C}$ obtained by requiring each of the above inequalities to be strict. 
Note that the Vandermonde  matrix~$(V_{i,j})=(\lambda_i^{j-1})$ with $0<\lambda_1<\dots< \lambda_n$ is in~$\TP_{2,C} $ as soon as $\lambda_{i+1}/\lambda_i >C $,
for all $1\leq i\leq n-1$, 
and therefore $\emptyset\ne\TP_{2,C} \subset\TN_{2,C}$ for every~$C\geq 1$.
The following theorem, which shows
that
\[ \TN_{2,C}\subset\TN\subset\TN_{2}
\]
for some value of $C$ depending on $n,m$, may
be thought of as an archimedean analogue
of Theorem~\ref{trop2}. Note that this theorem is
valid in particular for matrices with entries in the field
of real numbers.

\begin{theorem}
\label{TNc}
If $\puiseuxA\in(\TN_{2,C})^{n\times m}$ 
with $C\geq(\min(n,m)-1)^2$, then $\puiseuxA\in\TN$.
Similarly, 
under the same condition on $C$, 
if $\puiseuxA\in(\TP_{2,C})^{n\times m}$ then $\puiseuxA\in\TP$.
\end{theorem}

The proof of this result relies on a series of auxiliary results.

\begin{lem}\label{cyclebound}
Let $\puiseuxA\in(\TN_{2,C})^{n\times n}$.
Let $\gamma$ be any cyclic permutation of a subset $I$
of elements of $[n]$.
Then, 
 \begin{equation}\label{cycleb}\prod_{j\in I}\puiseuxA_{j,\gamma(j)}\leq \frac{1}{C^{|I|-1}} \prod_{j\in I}\puiseuxA_{j,j}\enspace.\end{equation}   
Moreover, the above inequality is strict as soon as~$\puiseuxA\in(\TP_{2,C})^{n\times n}$.
\end{lem}

\begin{proof}
We show that~\eqref{cycleb} holds for all cyclic permutations
of a subset $I\subset [n]$ by induction on 
the number of elements $k$ of this subset.
The base of the induction, $|I|=2$, holds
by definition of $\TN_{2,C}^{n\times n}$.

Consider now a cycle $\gamma=(i_1\ \dots\ i_k)$. A cycle
must have an inversion,
and we may assume without loss of generality that this inversion is on the indices $i_{k-1},i_k$, so that $i_{k-1}<i_k$ and $i_k>i_1$.
Considering the $\{i_{k-1},i_k\}\times \{i_1,i_k\}$
submatrix of $\puiseuxA$, we observe that
\[
\puiseuxA_{i_{k-1}i_k}\puiseuxA_{i_ki_1}\leq \frac{1}{C}\puiseuxA_{i_{k-1}i_1}\puiseuxA_{i_ki_k}
\]
It follows that
\[
\puiseuxA_{i_1,i_2} \puiseuxA_{i_2,i_3} \cdots  \puiseuxA_{i_{k-2},i_{k-1}} {(\puiseuxA_{i_{k-1},i_k}
\puiseuxA_{i_k,i_1})} \leq 
\frac{1}{C} (\puiseuxA_{i_1,i_2} \puiseuxA_{i_2,i_3}\dots \puiseuxA_{i_{k-1},i_1}) \puiseuxA_{i_ki_k}
\]
Applying the induction hypothesis to the cycle  $(i_1,\ldots,i_{k-1})$, we
obtain~\eqref{cycleb}.


The case of strict inequalities is obtained in the same way, using the fact
that all entries of a matrix in $\TP_{2,C}$ are positive.
\end{proof}



Consider a matrix $\puiseuxF\in\K^{n\times n}$, and define
the {\em maximal cycle mean} of $\puiseuxF$,
\[
\rho_{\max}(\puiseuxF):=\sup_{i_1\dots i_k}
\big|\puiseuxF_{i_1,i_2}\dots\puiseuxF_{i_k,i_1} \big|^{\frac{1}{k}} \enspace,
\]
where the maximum is taken over all sequences $i_1,\dots,i_k$ of distinct
elements of $\{1,\dots,n\}$.
The following result
gives a sufficient condition for
the determinant of a matrix to have the same sign as the determinant
of its diagonal.

\begin{thm}\label{th-indep}
Suppose that $\puiseuxA\in \K^{n\times n}$
can be written as 
\begin{align*}
\puiseuxA= \puiseuxD+\puiseuxB
\end{align*}
where $\puiseuxD$ is a diagonal matrix with non-zero diagonal
entries, and $\puiseuxB$ has zero diagonal entries.
If \begin{align}
 \rho_{\max}(\puiseuxD^{-1}\puiseuxB)<  1/(n-1)
\enspace,\label{e-rho}
\end{align}
then
\begin{align}
(\det\puiseuxD)^{-1}\det\puiseuxA\geq \Big(1-(n-1) \rho_{\max}(\puiseuxD^{-1}\puiseuxB)\Big)^n >0 \enspace,\label{e-newinequality}
\end{align}
in particular, $\det \puiseuxD \det \puiseuxA>0$. 
Moreover, if we only have a weak inequality
in~\eqref{e-rho}, then
$\det \puiseuxD \det \puiseuxA \geq 0$.
\end{thm}
\begin{proof}
A matrix $\puiseuxF\in \K^{n\times n}$ has $n$ eigenvalues in $\K[\sqrt{-1}]$,
counted with multiplicities, 
and we denote by $\rho(\puiseuxF)$ the spectral radius of $\puiseuxF$, i.e.,
the maximum of the moduli of these eigenvalues.
We denote by $\pat(\puiseuxF)$
the {\em Boolean pattern}
of the matrix $\puiseuxF$, so that the $(i,j)$-entry of
$\pat(\puiseuxF)$ is equal to $1$ if $\puiseuxF_{ij}\neq 0$,
and to $0$ otherwise. Let us recall the following inequality of Friedland~\cite{friedland}, relating the maximal cycle mean with the spectral radius:
\begin{align}
\rho(\puiseuxF) \leq \rho_{\max}(\puiseuxF) \rho(\pat(\puiseuxF)) \enspace .\label{e-shmuel}
\end{align}
Indeed, this inequality is established in~\cite{friedland} for matrices
with entries in $\R$. This property can be expressed
by a first order sentence in the language of ordered fields, and so,
by Tarski's completeness theorem mentioned in \Cref{subsec-nonarch}, 
it also holds for matrices with entries in $\K$.

Consider $\puiseuxF:=\puiseuxD^{-1}\puiseuxB$, so that $\puiseuxD^{-1}\puiseuxA= \puiseuxI +\puiseuxF$. 
Since $\puiseuxB_{ii}=0$, the matrix $\puiseuxP:=\pat (\puiseuxF)$,
which as $0/1$ entries, has at most $n-1$ entries equal to $1$ in each
row. Hence, $\rho(\puiseuxP)\leq \max_{i}\sum_j |\puiseuxP_{ij}|
\leq n-1$. Let $\puiseuxlambda_1,\dots,\puiseuxlambda_n\in
\K[\sqrt{-1}]$ denote the eigenvalues of $\puiseuxF$. 
We deduce
from~\eqref{e-shmuel} that if $ \rho_{\max}(\puiseuxF)<  1/(n-1)$, 
then $|\puiseuxlambda_i|\leq (n-1)\rho_{\max}(\puiseuxF)<  1$,
for all $1\leq i\leq n$.
Observe that 
\begin{align*}
(\det\puiseuxD)^{-1}\det\puiseuxA
&= \det (\puiseuxD^{-1}\puiseuxA)=
\prod_{1\leq i\leq n}(1+\puiseuxlambda_i)
\end{align*}
If $\puiseuxlambda_i$ is real, we have $1+\puiseuxlambda_i \geq 1-|\puiseuxlambda_i|\geq 1- (n-1)\rho_{\max}(\puiseuxF)>0$. Otherwise, we note that
both $\puiseuxlambda_i$
and its conjugate $\bar{\puiseuxlambda}_i$ are eigenvalues of $\puiseuxF$,
and observe that $(1+\puiseuxlambda_i)(1+\bar{\puiseuxlambda}_i)\geq (1-|\puiseuxlambda_i|)^2 \geq \big(1- (n-1)\rho_{\max}(\puiseuxF)\big)^2>0$.
Therefore, regrouping the eigenvalues 
by conjugate pairs, we deduce that
\begin{align}
\prod_{1\leq i\leq n}(1+\puiseuxlambda_i)
& 
\geq  
\Big(1-(n-1)\rho_{\max}(\puiseuxF)\Big)^n>0 \enspace,
\label{e-newineq}
\end{align}
which shows~\eqref{e-newinequality}.


If only the weak inequality holds in~\eqref{e-rho}, we can still
conclude that the weak inequality holds in~\eqref{e-newineq}.
\end{proof}

          \if{
The following result has appeared in several guises, it is
essentially one part of the tropical spectral theorem~\cite{BCOQ92,bcg},
and it is a special case of a non-linear Collatz-Wielandt theorem~\cite{nussbaum86}. The version we use is taken from~\cite{elsner}
\begin{lem}[Compare with~{\cite[Th.~7]{elsner}}]\label{lem-scaling}
For all $n\times n$ nonnegative matrices $\puiseuxB$, we have
\begin{align}
\rho_{\max}(\puiseuxB)= \inf\{\max_{i,j\in[n]}u_i^{-1}\puiseuxB_{ij}u_j\mid u\in \R^n_{>0}
 \}\enspace ,\label{e-scaling}
\end{align}
moreover, the infimum is attained if $\rho_{\max}(\puiseuxB)>0$.
\end{lem}
\begin{proof}
The equality in~\eqref{e-scaling} appeared in~\cite{elsner}, without the mention of the case in which
the infimum is attained. By making the change of variables $v_i:=\log u_i$, 
we deduce that
\[
\log \rho_{\max}(\puiseuxB)= \inf\{\lambda \mid \lambda\in \R, v\in \R^n,\; -v_i + \log \puiseuxB_{ij} + v_j \leq \lambda ,\;\forall i,j\in [n]\} \enspace.
\]
The latter expression is the value of a linear program, which, by a general
property, has an optimal solution (meaning
that the infimum is attained) as soon as this value,~$\log\rho_{\max}(\puiseuxB)$
here, is finite.
So, the infimum in~\eqref{e-scaling} is attained as soon as~$\rho_{\max}(\puiseuxB)>0$.
\end{proof}

Let us recall that a matrix~$\puiseuxA$ with entries
in $\R$, or $\K$, is \textbf{diagonally dominant}
 if~$\puiseuxA_{i,i}\geq \sum_{j\ne i}|\puiseuxA_{i,j}|,\ \forall  i$. 
The matrix $\puiseuxA$ is  \textbf{strictly diagonally dominant} if the latter inequalities are strict.

A key ingredient of the proof of Theorem~\ref{TNc} is the
following result, which gives a 
sufficient condition for a matrix to be diagonally similar
to a matrix with a dominant diagonal. For $\puiseuxB\in\R^{n\times n}$,
we denote by $|\puiseuxB|$ the matrix with entries
$(|\puiseuxB_{ij}|)$.
\begin{pro}\label{prop-scaling}
Suppose that $\puiseuxA\in \R^{n\times n}$
can be written as 
\begin{align*}
\puiseuxA= \puiseuxD+\puiseuxB
\end{align*}
where $\puiseuxD$ has positive diagonal entries and $\puiseuxB$
has zero diagonal entries.
If \begin{align}
 \rho_{\max}(\puiseuxD^{-1}|\puiseuxB|)\leq  1/(n-1)
\enspace,\label{e-rho}
\end{align}
then  there exists a diagonal matrix $\delta$ with positive diagonal
entries such that $\delta^{-1} \puiseuxA\delta$ is diagonally
dominant. In particular, $\det \puiseuxA\geq 0$.
Moreover, if the inequality in~\eqref{e-rho}
is strict, the matrix
$\delta^{-1} \puiseuxA\delta$ can be required
to be strictly diagonally dominant,
and $\det \puiseuxA>0$.
\end{pro}
\begin{proof}
Multiplying $\puiseuxA$ by $\puiseuxD^{-1}$ preserves
the diagonal dominance of $\puiseuxA$ and the sign of $\det \puiseuxA$,
hence, we may assume that the diagonal entries of $\puiseuxA$ are
equal to $1$.

Assume first that $\rho_{\max}(|\puiseuxB|)>0$. Then, the minimum
is attained in~\eqref{e-scaling} by some vector $u\in \R_{>0}^n$.
Setting~$\delta=\diag(u)$  and~$\tilde{\puiseuxA}=\delta^{-1}\puiseuxA\delta $, we get $\tilde{\puiseuxA}_{i,i}=1$, and 
\[ |\tilde{\puiseuxA}_{i,j}|
= u^{-1}_i|\puiseuxA_{i,j}|u_j\leq\rho_{\max}(|\puiseuxB|)\leq \frac{1}{n-1}
\]
It follows that $\sum_{j\neq i} |\tilde{\puiseuxA}_{ij}|\leq 1$, showing that
$\tilde{\puiseuxA}$ has a dominant diagonal.

Finally, if $\rho_{\max}(|\puiseuxB|)=0$, for all $\lambda>0$,
by~\eqref{e-scaling}, we can find a vector $u\in \R_{>0}^n$ such that 
$u^{-1}_i|\puiseuxA_{i,j}|u_j\leq \lambda$, for all $i,j\in [n]$.
Choosing any $\lambda\leq 1/(n-1)$, we conclude as above
that $\tilde{\puiseuxA}$ is diagonally dominant.
We also conclude that if $\rho_{\max}(|\puiseuxB|)< 1/(n-1)$, then
$\tilde{\puiseuxA}$ is strictly diagonally dominant.

It is known that a strictly diagonally dominant matrix $\tilde{A}$
satisfies $\det \puiseuxA>0$, this appears in particular
as the implication ``$M_{35}\Rightarrow A_1$''
in Theorem~2.3 of~\cite{berman}. Moreover, a immediate perturbation argument
implies that a diagonally dominant matrix $\tilde{A}$
satisfies $\det \puiseuxA\geq 0$,
(replace $\puiseuxA$ by $\puiseuxA+s I$, where $I$ is the identity matrix,
and let $s$ tend to $0^+$).
                  \end{proof}}\fi

\begin{proof}[Proof of Theorem~\ref{TNc}]
We need to show that every minor of $\puiseuxA$ is nonnegative (or positive).
Hence, possibly after replacing $\puiseuxA$ by a submatrix, 
we may assume that $n=m$. We also assume without loss of generality
that~$\puiseuxA_{i,i}=1$ for all $i\in [n]$. 
We set~$\puiseuxA=\puiseuxB+\puiseuxI$, where~$\puiseuxI$ is the identity matrix, and $\puiseuxB$ is the off-diagonal part of $\puiseuxA$
meaning that $\puiseuxB_{i,j}= \puiseuxA_{ij}$
if $i\neq j$ and $\puiseuxB_{i,i}=0$. 
Lemma~\ref{cyclebound} yields the following bound
on the maximal cycle mean of $\puiseuxB$,
\begin{align}\label{e-boundrho}
\rho_{\max}(\puiseuxB)\leq\max_{2\leq k\leq n}\frac{1}{(C^{k-1})^{\frac{1}{k}}}
= \frac{1}{C^{\frac{1}{2}}} \enspace .
\end{align}

If  $C\geq(n-1)^2$, $\rho_{\max}(\puiseuxB)\leq 1/(n-1)$, and we deduce
from \Cref{th-indep} that $\det \puiseuxA\geq 0$.
Similarly, if $\puiseuxA\in \TP_{2,C}$, we deduce along the
same lines that $\det(\puiseuxA)> 0$.
\end{proof}
\begin{rem}
The statement of Theorem~\ref{TNc} is inspired by 
an analogous result for {\em symmetric} matrices
which has appeared in~\cite{issac2016}.
It is shown there that if $\puiseuxA$ is an $n\times n$
symmetric matrix with nonnegative diagonal entries, such that
$\puiseuxA_{i,i}\puiseuxA_{j,j}\geq 
(n-1)^2\puiseuxA_{i,j}^2$ for all $i<j$,
then, $\puiseuxA$ is positive semidefinite. 
\end{rem}
\if{
\begin{rem}\label{rk-friedland}
An alternative approach to \Cref{prop-scaling} is the following condition,
which involves the spectral radius of a nonnegative matrix,
$\rho(\cdot)$, instead of the maximal geometric mean $\rho_{\max}(\cdot)$:
{\bf Claim}.\/ {\em Assume that $\puiseuxA= \puiseuxD+\puiseuxB$ where $\puiseuxD$ has positive diagonal entries, and $\rho(\puiseuxD^{-1}|\puiseuxB|)<1$. Then, there exists a diagonal
matrix $\delta$ with positive diagonal entries such that $\delta^{-1} \puiseuxA\delta$ is strictly diagonally dominant}.
To show this, we proceed as in the
proof of \Cref{prop-scaling}, but, instead of using~\Cref{lem-scaling}, we
use the classical Collatz-Wielandt result (see e.g.~\cite{nussbaum86}), 
showing that for all nonnegative matrices~$\puiseuxF\in \R^{n\times n}$,
\begin{align}
\rho(\puiseuxF)= \inf \{ \lambda>0\mid \exists u\in \R_{>0}^n,\; \puiseuxFu\leq \lambda u\} \enspace .\label{e-cwnew}
\end{align}
If $\rho(\puiseuxD^{-1}|\puiseuxB|)<1$ we 
deduce that there is $0<\lambda<1$ and
$u\in \R_{>0}^n,\; \puiseuxD^{-1}|\puiseuxB|u\leq \lambda u$,
and setting $\delta =\operatorname{diag}(u)$, we conclude
that $\delta^{-1}\puiseuxA\delta$ is strictly diagonal dominant,
which shows the claim.
Moreover, Friedland showed in~\cite{friedland} that for a nonnegative matrix~$\puiseuxF$,
\[ \rho(\puiseuxF)\leq \rho_{\max}(\puiseuxF) \rho(\pat(\puiseuxF))\enspace,
\]
where $\pat(\cdot)$ denotes the Boolean pattern of a matrix,
i.e., the $0/1$ matrix with $(i,j)$ entry equal to $1$ whenever $\puiseuxF_{ij}>0$. Applying this result to the situation of \Cref{prop-scaling}
in which $\puiseuxB$ has $0$ diagonal entries,
so that $\rho(\pat(\puiseuxD^{-1}|\puiseuxB|))\leq n-1$, 
we deduce that $\rho(|\puiseuxB|)\leq \rho_{\max}(\puiseuxB)(n-1)$,
so that the condition that $\rho(\puiseuxD^{-1}|\puiseuxB|)<1$ in the above claim is milder than the condition $\rho_{\max}(\puiseuxD^{-1}|\puiseuxB|)<1/(n-1)$ in \Cref{prop-scaling}. 
However, this alternative approach is less satisfactory in the case of weak
inequalities: unlike in \Cref{lem-scaling}, the infimum may not be attained
in \Cref{e-cwnew} even if $\rho(|\puiseuxF|)>0$. In fact,
the condition that $\rho(|\puiseuxB|)\leq 1$ does not imply that $\puiseuxA$
can be reduced to a (non-strictly) diagonally dominant matrix (take $\puiseuxA = \left(\begin{smallmatrix} 1 & 1\\0 & 1\end{smallmatrix}\right)$).
\end{rem}}\fi
We now introduce the tropical notion of diagonal dominance,
following the terminology of~\cite[\S3]{MA}:
a matrix~$A\in\R_{\max}^{n\times n}$ is  \textbf{tropical diagonally dominant}
if $\per(A)=\sum_{i\in[n]} A_{i,i},$ and it is \textbf{tropical strictly diagonally dominant} if $\per(A) >\sum_{i\in [n]}A_{i,\sigma(i)}$ holds for all
permutation $\sigma$ different from the identity. 
We denote by~$\DD$ the set of matrices such that every 
submatrix is tropical diagonally dominant, and by~$\NDD$ 
the set of matrices such that every submatrix is tropical strictly
diagonally dominant.

\begin{pro} \label{DD}
$\ \TN^\trop\subset\DD\ $ and $\ \TP^\trop\subset\NDD$.
\end{pro}

\begin{proof} 
We show that if $A\in \TN^\trop$ is $n\times n$, then  the identity
permutation attains the permanent of $A$. 
We have trivially $\TN^\trop\subset \TN^{\trop}_2$, and we observed
in the proof of \Cref{2trop2} that if $A\in\TN^{\trop}_2$,
then  the canonical lift $(t^{A_{ij}})$ belongs to $\TN_2(\K)$.
The latter set coincides with $\TN_{2,C}(\K)$,
for $C=1$.  
It follows from \Cref{cyclebound} that for all cycle
$\gamma$ with set of elements $I$,
\begin{align}
\label{e-interm}
\prod_{j\in I}t^{A_{j,\gamma(j)}}
\leq \prod_{j\in I}t^{A_{j,j}}\enspace .
\end{align}
Thus
the weight of any permutation of $A$ is dominated
by the weight of the identity permutation.

The case in which $A\in \TP^\trop$ follows by noting
that the inequalities~\eqref{e-interm} are strict.

\end{proof}


It follows from \Cref{subsequent0} that 
a real matrix~$A$ is 
in~$\TN^{\trop}(\R)$ (resp.~in~$\TP^{\trop}$) if and 
only if its~$2\times 2$ solid minors are 
tropically nonnegative (resp.~tropically positive).
\if{
\begin{lem}\label{subsequent} 
 \end{lem}

\begin{proof} The ``only if'' part of the lemma is clear.
Suppose $A\in \TN^\trop$.
We consider consecutive rows~$i,i+1$, and show,
by induction on $t\geq 1$, that the $\{i,i+1\}\times \{j,j+t\}$
minor of $A$ is tropically nonnegative. The base case $t=1$
is valid by our assumption. 
Assuming that the property holds for~$t=k$, we get, 
$$a_{i,j}\odot \underline{a_{i+1,j+k}}\odot a_{i+1,j+k+1}\geq a_{i+1,j} 
\odot a_{i,j+k}\odot  a_{i+1,j+k+1}\geq  a_{i+1,j}\odot a_{i,j+k+1}\odot 
\underline{a_{i+1,j+k}}\enspace ,$$
and canceling the underlined terms, we get
$\ a_{i,j}\odot a_{i+1,j+k+1}\geq a_{i,j+k+1}\odot a_{i+1,j},$
which shows that the property holds 
for~$t=k+1$.
Moreover, we deduce, by a similar induction on $s\geq 1$, that
the $\{i,i+s\}\times \{j,j+t\}$
minor of $A$ is tropically nonnegative.

The proof in the case of ~$A\in\TP^\trop$ is an immediate variant,
replacing weak inequalities by strict ones.
\end{proof}}\fi
Therefore,
for an $n\times m$ matrix 
to be tropical totally nonnegative (resp.~positive)
with real entries, it suffices that the following $nm$ conditions
hold:
each of the entries of the first row and first column
is real and all the~$2\times 2$ solid submatrices satisfy the
Monge property (resp.~strict Monge property). 
Note that the conclusion of \Cref{subsequent0} does not carry
over to the situation in which~$A$ has~$\zero$ entries, as seen for 
instance by
$A=\left(\begin{smallmatrix}2&\minusinfty&2\\2&\minusinfty&0\end{smallmatrix}\right).$

Classically, there are~$nm$ sufficient conditions for an~$n\times m$  
matrix to be totally positive, given by the positivity of its \textit{initial minors} (solid and bordering either the left or the top edge of the matrix).  
In the following proposition, we provide the tropical analogue of this property.
\begin{pro}\label{initial} Let~$A\in\R^{n\times m}$. The following are equivalent:
\begin{enumerate}
\item\label{m1} The solid~$2\times 2$ tropical minors of~$A$
 are tropically positive,
\item \label{m2} The tropical initial  minors of~$A$ are tropically positive,
\item \label{m3} $A\in\TP^\trop$.
\end{enumerate} 
\end{pro}

\begin{proof}
We already proved that \eqref{m3} and~\eqref{m1} are equivalent
(Lemma~\ref{subsequent0}). Moreover, \eqref{m3} trivially implies \eqref{m2}.


We show that~\eqref{m2} implies~\eqref{m1}.  
If the tropical initial minors of~$A$ are tropically positive,
then the initial minors of any lift of~$A$ are positive.
Thus, any lift $\puiseuxA$ of~$A$ is in~$\TP$, 
and in particular, the canonical lift $\puiseuxA:=(t^{A_{ij}})$ is in $\TP$.
It follows that every $2\times 2$ minor of $\puiseuxA$ is positive, 
which implies, since the lift is canonical,
that every $2\times 2$ minor of $A$ is tropically positive.
\end{proof}

As in the classical case, $A$ is not necessarily in $\TN^\trop$, if we allow (sign-)singular initial minors. See for instance the matrix  
$\left(\begin{smallmatrix} 
 1 & 1 & 1\\
 1 & 1 & 3\\
 2 & 2 & 1\\\end{smallmatrix}\right),$ which has classically and tropically nonnegative initial minors,
however its $2\times 2$ bottom-right minor is classically and tropically negative.



\begin{pro}\label{valuesubsequent} A  tropical totally nonnegative 
real matrix is uniquely determined by the values of its~$2\times 2$ 
solid minors and by the entries on its first row and first column.

\end{pro}

\begin{proof} 
 Denote  by~$M_{i,j}$ the permanent of the~$2\times 2$ 
submatrix  of consecutive rows~$\{i,i+1\}$ and consecutive columns~$\{j,j+1\}$, 
of the matrix~$A\in \TN^\trop$. We have 
$ A_{i-1,j-1} \odot A_{i,j} = M_{i-1,j-1}$, 
or $A_{i,j}=A_{i-1,j-1}^{\odot -1}\odot M_{i-1,j-1}$.
This implies, after an immediate induction, that $A$ is well determined by the value of the $M_{ij}$, and by its first row and column.

\end{proof}



\section{Double echelon and staircase forms}\label{desm} 
In this section, we describe the ``shape'' of tropical totally nonnegative matrices. We first consider the $\zero$/non-$\zero$ pattern,
showing that, as in the classical case (\cite{Fallat}, \cite[\S1.6]{Fallat&Johnson}, the tropical matrices which have non-$\zero$
rows and columns, have a double echelon form.
Next, we characterize the set of tropical totally nonnegative
matrices with {\em finite} entries. The latter constitute
a polyhedron, which coincides with the set of (anti)-Monge matrices.
Then, we deduce from a known characterization of 
Monge matrices~\cite[\S~2]{BKR},\cite{FIEDLER}
that every tropical totally nonnegative real
matrix is amenable to a certain ``staircase form''
by diagonal scaling.



\subsection{Double echelon form}
Recall that a \textit{Boolean matrix} is a 
matrix all of whose entries  are~$0$ and~$1$, and that the \textit{Boolean pattern} of a matrix~$A$ over a field, already considered in the proof
of \Cref{th-indep} is the Boolean matrix~$B$ such 
that~$B_{i,j}=1$ if~$A_{i,j}\ne 0$, and~$0$ otherwise.
\begin{df} \label{DED}
An~$n\times m$ Boolean  matrix  has a \textbf{double echelon pattern} if
\begin{enumerate}
\item\label{cond1} Every row is in one of the following forms
\begin{enumerate}
\item $(1,\dots,1)\enspace ,$
\item $(1,\dots,1,0,\dots,0)\enspace,$
\item $(0,\dots,0,1,\dots,1,0,\dots,0)\enspace,$
\item $(0,\dots,0,1,\dots,1)\enspace,$
\end{enumerate}

\item\label{cond2} The first and last nonzero entries in row~$i$ 
appear not to the left of  
the  first and last nonzero entries in row~$i-1$ 
respectively, $\forall i=2,\dots,n.$
\end{enumerate}
A  matrix over~$\R$ or~$\K$ is in \textbf{double echelon form} if it has 
a double echelon Boolean pattern.
A  matrix over~$\rmax$ is in {double echelon form} if its lifts have a 
double echelon Boolean pattern.
\end{df}

\begin{exa}\label{DEE} Every lift   of
$$E=\left(\begin{array}{ccccc}
1&1&3&\minusinfty&\minusinfty\\ 
\minusinfty&2&1&2&\minusinfty\\
\minusinfty&1&1&1&3\\
 \minusinfty&\minusinfty&1&1&1
\end{array}\right)$$
 have the double echelon  pattern
$$\left(\begin{array}{ccccc}
1&1&1&0&0\\ 
0&1&1&1&0\\
0&1&1&1&1\\
 0&0&1&1&1
\end{array}\right).$$
\end{exa}

\begin{pro}\label{TNDE} A tropical totally nonnegative  matrix~$A$  
with no $\zero$  row or column is in double echelon form. \end{pro}

\begin{proof} 
Suppose first, by contradiction, that $A$ does not satisfy condition~\eqref{cond1}. Then there 
exists~$ i\in[n]$ such that~row~$A_i$ does not have a Boolean pattern as in~(a)--(d). That 
is, there exist~$ t<k<\ell$ such that  $A_{i,t}\ne\zero,\ A_{i,k}=\zero$ and $A_{i,\ell}\ne\zero$.  Since $A$ does not include zero columns, there exists row $j$ such that $A_{j,k}\ne\zero$, and therefore
the $\{i,j\}\times\{k,\ell\}$ minor is tropically negative.

Suppose now that $A$ satisfies condition~\eqref{cond1} but not condition~\eqref{cond2}. Then there 
exists~$ i$ such that the first (or resp.~last) non-$\zero$   entry~$A_{i,j}$ in row~$i$ 
appears  to the left of   the  first (or resp.~last) non-$\zero$   entry~$A_{i-1,k}$ in row~$i-1$.
As a result, the $\{i-1,i\}\times\{j,k\}$ minor is tropically negative.
\end{proof}

The following corollary is a straightforward consequence of~\Cref{TNDE}, since a matrix with no~$\zero$ maximal solid minor cannot have a
$\zero$-row and a no $\zero$-column.
\begin{cor}\label{DE} 
 If all the maximal solid minors 
of~$A\in (\TN^{\trop})^{n\times m}$ are non-$\zero$,
then~$A$ is in double echelon form.
\end{cor}


\subsection{Staircase  matrices}\label{sec-staircase}
For all $(i,j)\in ([n]\setminus\{1\})\times (j\in[m]\setminus\{1\})$, 
we define the \textbf{elementary staircase matrix} $L^{(i,j)}\in \R^{n\times m}$
as follows:

\[ L^{(i,j)}_{t,s}=
\begin{cases}1&\text{ if } (t,s)\in[n]\setminus[i-1]\times[m]\setminus[j-1]\\0&
\text{otherwise }\end{cases}\enspace .
\]
For instance, for $m=n=3$,
\[
L^{(2,3)} = \left(\begin{array}{ccc}
0 & 0 & 0\\
0 & 0 & 1\\
0 & 0 & 1
\end{array}\right)
\]
A matrix~$A\in\mathbb{R}^{n\times m}$ is \textbf{a staircase matrix} 
if it can be written as a nonnegative linear combination
of elementary staircases matrices.
We shall use the following characterization which was shown in~\cite{FIEDLER};
earlier characterizations of the same nature appeared
in~\cite[Lemma~2.1]{BKR}. 
\begin{thm}[{\cite[Theorem~3.3]{FIEDLER}}]\label{th-fiedler}
A matrix~$A\in \R^{n\times m}$ is a Monge matrix if and only
if there is a staircase matrix $S\in \R^{n\times m}$
and two vectors $u \in \R^n$ and $v\in \R^m$ such
that 
\[
A_{ij}=S_{ij}+u_i+v_j 
\]
\end{thm}

The following result follows readily by combining Theorem~\ref{trop2},
showing that the matrices $\TN^\trop(\mathbb{R})$ are precisely
the Monge matrices, with Theorem~\ref{th-fiedler}.
\begin{thm} \label{SC}  
A matrix~$A$ is in~$\TN^\trop(\mathbb{R})$ if and only if 
there exist tropical diagonal matrices~$D,D'$ such that~$D\odot A\odot D'$ 
is a staircase matrix.\end{thm}
\begin{exa}
The following factorization illustrates this result:
\begin{align*}
\left(\begin{array}{ccc}0&\minusinfty&\minusinfty\\\minusinfty&1&\minusinfty\\\minusinfty&\minusinfty&2
\end{array}\right)\odot
\overbrace{
\left(\begin{array}{ccc}1&0&-1\\
0&1&0\\-1&0&1\end{array}
\right)
}^{\text{in}\ \TN^\trop(\R)}
\odot\left(\begin{array}{ccc}-1&\minusinfty&\minusinfty\\
\minusinfty&0&\minusinfty\\\minusinfty&\minusinfty&1
\end{array}\right)&=
\left(\begin{array}{ccc}0&0&0\\0&2&2\\0&2&4\end{array}\right)\\
&=2L^{(2,2)}+2L^{(3,3)}\enspace .\end{align*}
\end{exa}

\Cref{SC} provides a polyhedral characterization
of $\TN^\trop(\mathbb{R})$. 
\begin{cor}\label{prop-new}
The set $\TN^\trop(\mathbb{R})$, thought of as a subset
of $\R^{n\times m}$, is a polyhedron which can be written
as the Minkowski sum $\mathsf{V}+\mathsf{S}$, where
\[
\mathsf{V}=\{(d_i+d'_j)_{i\in [n],j\in[m]}, d\in \R^n,d'\in \R^m\}
\]
is the $n+m-1$ dimensional lineality space of this polyhedron, 
and $\mathsf{S}$ is the set of staircase matrices, which
is a simplicial cone of dimension $(n-1)(m-1)$ with extreme
rays generated by the elementary staircase matrices $L^{(i,j)}$, 
for 
$(i,j)\in 
([n]\setminus\{1\})\times (j\in[m]\setminus\{1\})$. 
\end{cor}
\begin{proof}
The decomposition of \Cref{SC} shows that
$\TN^\trop(\R)=\mathsf{V}+\mathsf{S}$.
Since $\mathsf{S}$ contains no affine line, 
$\mathsf{V}$ must coincide with the lineality space
of $\TN^\trop(\R)$. 
It also follows from \Cref{SC} that $\mathsf{S}$ is 
precisely the set of Monge matrices with 
a zero first row and a zero
first column. Each matrix in this set can be written in a unique
way as a positive linear combination of elementary
staircase matrices, indeed,
\begin{align}\label{lcsc}A&=\ \sum_{i=3}^{n}\sum_{j=3}^{m}
  (A_{i,j}+A_{i-1,j-1}-A_{i-1,j}-A_{i,j-1})L^{(i,j)}+\\
&\qquad \qquad \sum_{j=3}^{m}(A_{2,j}-A_{2,j-1})L^{(2,j)}+\sum_{i=3}^{n}
 (A_{i,2}-A_{i-1,2})L^{(i,2)}+A_{2,2}^{(2,2)} L^{(2,2)}
\nonumber
\end{align}
This implies that $\mathsf{S}$ is the $(n-1)(m-1)$-dimensional
simplicial
cone generated by the matrices $L^{(i,j)}$, 
$(i,j)\in ([n]\setminus\{1\})\times ([m]\setminus\{1\})$. 
Observe that the dimensions of $\mathsf{S}$ and $\mathsf{V}$ sum
to $nm$, in accordance with $\TN^\trop(\R)$ being full
dimensional (it has non empty interior).
\end{proof}




\begin{rem}\label{rem-newbis}
Dually, the collection of inequalities
\[
%
A_{i,j}+ A_{i+1,j+1}\geq A_{i,j+1}+ A_{i+1,j},
\]
for $1\leq i\leq n-1$ and $1\leq j\leq m-1$ define
$\TN^\trop(\mathbb{R})$ (by \Cref{subsequent0}). There are
$(n-1)(m-1)$ inequalities of this kind. This collection
of inequalities is minimal. Otherwise, we would eliminate
some of these inequalities, and end up with a representation
of the set of staircase matrices, identified to a 
polyhedral cone of $\R^{(n-1)\times(m-1)}$, by fewer than $(n-1)(m-1)$ inequalities. This is absurd, since it follows from \Cref{prop-new} that this polyhedral
cone is simplicial of dimension $(n-1)(m-1)$, so its number of facets,
which coincides with the minimal cardinality of a defining set of inequalities,
is also $(n-1)(m-1)$.
\end{rem}






\section{Matrix factorization and lifts of~$\TP^\trop$and~$\TN^\trop$}\label{factlift}  
In this section, we relate
the tropical of matrix classes with the image by the valuation of
the corresponding matrix classes over a nonarchimedean field.
In particular, 
we find a tropical analogue to 
the Loewner--Whitney theorem,
which appeared in Loewner's work~\cite{Loewner},
being deduced there from a result of Whitney~\cite{Whitney}.
This theorem shows that any invertible totally nonnegative
matrix is a product of nonnegative elementary Jacobi matrices
(see for  instance~\cite[Theorem~12]{F&Z}).

\begin{df} An 
(elementary)
\textbf{Jacobi matrix} is an invertible matrix that 
differs from the identity matrix in at most one entry
located either on the main diagonal or immediately above
or below it. Analogously, an 
(elementary) \textbf{tropical Jacobi matrix} differs from 
the tropical identity matrix in at most one entry, which must be finite
and located either on the main diagonal or immediately above
or below it.
 \end{df}
This definition follows the one
of notion of  
~\cite{F&Z}, and corresponds to 
the~$LU$ factorization in~\cite{Fallat&Johnson}.

\begin{theorem}
\label{rltn} Over $\ \R_{\max}^{n\times m}$ we have 
\begin{equation}\label{rqtrl}
\TP^{\trop}\subsetneqq 
\val(\TP)\subsetneqq\val(\TN) \subset\TN^{\trop},\ \  \forall n,m\enspace,\end{equation}  
and when~$ n=m$ we  have
\begin{equation}\label{sqrl}\val(\TP^{n\times n})\subsetneqq\val(\GL_n\cap\TN)=
\langle\text{tropical~Jacobi~matrices}\rangle\subsetneqq
\val(\TN^{n\times n}) \enspace,\end{equation}
where $\langle \cdot \rangle$ denotes the semigroup generated
by a set of matrices.
\end{theorem}

\begin{proof}  We consider the canonical lift~$\puiseuxA$ of~$A\in\TP^{\trop}$, and note that 
$\val : \K_{\geq 0}\rightarrow \rmax$ is a
morphism of semifields. 
For a~$d\times d$ submatrix~$M$ 
in~$A$, indices naturally ordered, with the  
corresponding  submatrix~$\puiseuxM=(t^{M_{i,j}})$ 
in~$\puiseuxA$,  we have $$\det(\puiseuxM)=\sum_{\sigma\in S_d}
\sign(\sigma)\prod_{i\in [d]}t^{M_{i,\sigma(i)}}=
\sum_{\sigma\in S_d}\sign(\sigma)t^{\sum_{i\in [d]}
M_{i,\sigma(i)}}\enspace .$$ 
Since  the permutations 
of maximal weight in~$\per(M)$ are even, the monomials
with maximal valuation in the latter expansion have a positive sign,
and so
$\det(\puiseuxM)\in\mathbb{K}_{>0}$, 
which shows that 
$\TP^{\trop}\subset \val(\TP)$. Note also in passing that the
same proof shows that any lift $\puiseuxA$ of $A$ is totally positive.

The first inclusion in \eqref{rqtrl} is strict since  $$\puiseuxA=\left(\begin{array}{cc}2t&t\\
1&1\end{array}\right)\in \TP\ \ \text{ but }\ \ \val(\puiseuxA)=
\left(\begin{array}{cc}1&1\\0&0\end{array}\right)\notin \TP^\trop\enspace .$$
The second inclusion in~\eqref{rqtrl},
as well as the inclusions in~\eqref{sqrl}, are trivial 
whereas the third inclusion in~\eqref{rqtrl}
follows from \Cref{2trop2} and \Cref{trop2}.

The inclusions in~\eqref{sqrl}, and the second inclusion
in~\eqref{rqtrl}, are
strict since 
$$\val\left(\begin{array}{cc}1&0\\0&1\end{array}\right)
\in\val(\GL_2\cap\TN)\setminus\val(\TP)\ \text{ and }\ 
\val\left(\begin{array}{cc}1&0\\0&0\end{array}\right)
\in\val(\TN)\setminus\val(\GL_2\cap\TN)\enspace.$$

The Loewner-Whitney theorem shows that $\GL_n\cap\TN=\langle\text{nonnegative Jacobi matrices}\rangle$. Observe that the valuation sends
the nonnegative Jacobi matrices with entries in $\mathbb{K}$ to
the tropical Jacobi matrices.
Since~$\val: \mathbb{K}_{\geq 0}\rightarrow \rmax$ 
is a morphism of semifields, we deduce
that the image by the valuation of the semigroup 
$\langle\text{nonnegative Jacobi~matrices}\rangle$
is included in the semigroup $\langle\text{tropical~Jacobi~matrices}\rangle$.
Conversely, $\val\langle\text{nonnegative Jacobi matrices}\rangle\supset\langle
\text{tropical~Jacobi~matrices}\rangle$ since every tropical Jacobi 
matrix~$J$ can be trivially lifted to a Jacobi matrix~$\puiseuxJ$ 
over~$\mathbb{K}$. Thus, if~$A=J_1\odot\cdots\odot J_k$, then~$\puiseuxA$ 
such that~$\val(\puiseuxA)=A$ may be defined by~$\puiseuxJ_1
\cdots \puiseuxJ_k$, where~$\val(\puiseuxJ_i)=J_i\enspace\forall i$. 
\end{proof}

\begin{theorem}
\label{rltn2} A matrix~$A$ is in~$\TP^{\trop}$ if and only if every lift of $A$ is 
in $\TP$. 
\end{theorem}

\begin{proof}
We already showed in the initial part of the proof 
of \Cref{rltn} that if $A$ is in $\TP^{\trop}$,
any lift of $A$ is in $\TP$.

We now show the ``if"  part of the statement.
We assume~$A\notin \TP^\trop$ and find a lift not in~$\TP$.
Since, by \Cref{trop2}, $\TP^\trop=\TP^\trop_2$, we
deduce that either $A$ has an infinite entry, or that 
$A$ has a $2\times 2$ submatrix, say its $\{i_1,i_2\}\times \{j_1,j_2\}$
submatrix, with a tropically nonpositive determinant.
In the former case, every lift of $A$ has a $0$ entry, and so
no lift of $A$ can be totally positive. In the latter case, we
have $A_{i_1,j_2}+A_{i_2,j_1}\geq A_{i_1,j_1}+ A_{i_2,j_2}>\minusinfty$. 
Choose now any lift $\puiseuxA$ of $A$ of the form
$\puiseuxA_{i,j}= b_{i,j}t^{A_{i,j}}$, with
$b_{i,j}>1$ for $(i,j)\in \{(i_1,j_2),
(i_2,j_1)\}$ and $b_{i,j}=1$ otherwise.  It is immediate
that $\puiseuxA_{i_1,j_2}\puiseuxA_{i_2,j_1}>\puiseuxA_{i_1,j_1}\puiseuxA_{i_2,j_2}$,
showing that $\puiseuxA$ is not totally positive.
\end{proof}

An inspection of the  proof above suggests the following
more general construction. It will be convenient to denote by $*$ the Hadamard
product (i.e., entrywise product) of two matrices. 
Let $\puiseuxB\in \K^{n\times m}$, $A\in \rmax^{n\times m}$, consider the canonical
lift $t^A:=(t^{A_{ij}})$,  together with the Hadamard product
$\puiseuxA:= \puiseuxB * t^A$, i.e.,
\begin{align}
\puiseuxA_{i,j}:=\puiseuxB_{ij}t^{A_{ij}}
\label{e-def-hadamard}\end{align}
Observe that if $\puiseuxB\in \Kpos^{n\times m}$ with
$\val \puiseuxB_{ij} =0$ for all $i,j$,
which is the case in particular if $\puiseuxB\in \R_{>0}^{n\times m}$,
then, $\puiseuxA$ is a lift of $A$.

One may ask whether for~$\puiseuxB\in \TN^{n\times m}$,
and~$A\in(\TN^{\trop})^{n\times m}$,
the matrix obtained by the Hadamard 
product construction~\eqref{e-def-hadamard}
is in~$\TN(\mathbb{K})$. 
The example of
$$A=\left(\begin{array}{ccc}0&0&\minusinfty\\0&0&0\\\minusinfty&0&0\end{array}
\right)\in \TN^\trop\ \text{ and }\ \puiseuxB=\left(\begin{array}{ccc}1&1&1\\
1&1&1\\1&1&1\end{array}\right)\in \TN\enspace,$$ 
leading to $\det \puiseuxA=-1$, shows
that this is not necessarily true.
The following result shows, however, that
the conclusion becomes true if 
$\puiseuxB\in \TN_{2,C}(\K)$
for a suitable constant $C$. This will also provide 
a $\TN(\K)$ lift
for every $\TN^\trop$ matrix. 

\begin{theorem}
 \label{eqrltn} Let~$A\in(\TN^{\trop})^{n\times m}$.
For~$\puiseuxB\in \TN_{2,(n-1)^2}(\K)$, 
 the Hadamard product matrix~$\puiseuxA\in\mathbb{K}^{n\times m}$ defined 
in~\eqref{e-def-hadamard}
belongs
to
~$\TN(\mathbb{K})$.
Moreover, if~$\per(A)\ne \zero$ and $n=m$, 
then 
$\puiseuxA\in\GL_n\cap\TN(\mathbb{K})$.
If, in addition, $A\in\TN^{\trop}(\mathbb{R})$,
then~$\puiseuxA\in\TP(\mathbb{K})$.
\end{theorem}

To show this theorem, we make the following immediate observation.

\begin{lem}\label{hdmrd}
Given $C_1,C_2\geq 1$, if $M_1\in\TN_{2,C_1}$ and $M_2\in\TN_{2,C_2},$ then the Hadamard product~$M_1*M_2$ is in $\TN_{2,C_1C_2}$. \hfill\qed
\end{lem}


\begin{proof}[Proof of Theorem~\ref{eqrltn}]
By Lemma~\ref{hdmrd}, since $A\in \TN^\trop\subset \TN^\trop_2$,
it is immediate that 
the canonical lift of $A$ is in $\TN_2(\K)=\TN_{2,1}(\K)$, 
and therefore, by \Cref{TNc}, 
$$\puiseuxA\in\TN_{2,(n-1)^2\cdot 1}(\K)\subset
\TN_{2,(\min(n,m)-1)^2}(\K)
\subset \TN(\K).$$

Suppose now that $n=m$, $\puiseuxB\in \TP_{2,(n-1)^2}(\R)$ and that $\per(A)\neq \zero$.
We saw in~\Cref{DD} that $A$ is tropical diagonally dominant,
in particular $A_{ii}\neq \minusinfty$, for all $i\in [n]$.
After subtracting $A_{ii}$ to the $i$th row of $A$, and dividing
the $i$th row of $\puiseuxB$ by $\puiseuxB_{ii}$, we may
assume that $A_{ii}\equiv 0$ and $\puiseuxB_{ii}\equiv 1$,
so that $\puiseuxA_{ii}\equiv 1$. Since $\puiseuxB\in \TP_{2,(n-1)^2}$,
we have $\puiseuxB\in \TP_{2,C}$ for some $C<(n-1)^2$. We can
write, as in \Cref{th-indep},
\[
\puiseuxA = \puiseuxI + \puiseuxB'
\]
where $\puiseuxI$ is the identity matrix, and $\puiseuxB'$ has zero
diagonal entries. From $\puiseuxB\in \TP_{2,C}$, we deduce, 
using~\eqref{e-boundrho}, that $\rho_{\max}(\puiseuxB')\leq 1/C^{1/2}$.
Therefore, $\rho_{\max}(\puiseuxA)\leq 1/C^{1/2}<1/(n-1)$. Then, it
follows from \Cref{th-indep} that $\det \puiseuxA>0$. In particular,
$\puiseuxA\in \GL_n$.

Finally, if~$A\in\TN^{\trop}(\mathbb{R})$ 
and $\puiseuxB\in \TP_{2,(n-1)^2}$, then, using \Cref{TNc} again,
we deduce that~$\puiseuxA\in\TP_{2,(n-1)^2}(\mathbb{K})\subset\TP(\mathbb{K})$.
\end{proof}

\begin{rem}
Let $\puiseuxA$ and $A$ be as in \Cref{eqrltn}.
For all subsets $I\subset [n]$, $J\subset[m]$ with the same
cardinality, we denote by $\puiseuxA_{I,J}$ the $I\times J$ submatrix
of $\puiseuxA$, and use a similar notation for the matrix $A$.
Then, we note that if $A\in \TN^\trop(\R)$, and if 
$\puiseuxB\in \TN_{2,(n-1)^2}(\K)$ is choosen so that $\val \puiseuxB_{ij}$ is identically $0$, then,
\begin{align}
\val \det\puiseuxA_{I,J}   = \per A_{I,J} \enspace .
\label{e-identnew}
\end{align}
\end{rem}

\begin{rem}\label{lem-old}
Property~\eqref{e-identnew} may not hold if we choose a lift
different from the one of \Cref{eqrltn}.
For example 
$$A:=\left(\begin{array}{cc}
   0&0\\
 0&0 
\end{array}\right)=\val\puiseuxA, \qquad
\text{where}\qquad \puiseuxA:=\left(\begin{array}{cc}
  1&1\\
1-t^{-1}&1+t^{-1} 
\end{array}\right),$$
but the tropical permanent $\per A=0$ differs from $\val\det\puiseuxA=2t^{-1}$.
\end{rem}

We get the following corollary of \Cref{eqrltn}.
\begin{theorem}
\label{coro-reverse}
We have $\TN^\trop = \val (\TN)$. \hfill\qed
\end{theorem}
\begin{proof}
We showed in \Cref{rltn} that $\val(\TN) \subset \TN^\trop$. Conversely,
suppose that $A\in \TN^\trop$. Let us $C:=(n-1)^2$, and
take any matrix $B\in \TN_{2,C}(\R)\cap \R_{>0}^{n\times m}$ (we already observed that such matrices
exists for all $C$, for instance, $B$ may be a Vandermonde matrix). Then,
it follows from \Cref{eqrltn} that $\puiseuxA:=B*t^A\in \TN(\K)$,
and $\val(\puiseuxA)=A$. 
\end{proof}
The following is an immediate consequence of \Cref{coro-reverse}.
\begin{corollary}\label{coro-fin}
We have $\TN^\trop (\R)= \val (\TN(\K^*))$. 
\end{corollary}

\begin{rem}\label{rem-new2}
Together with \Cref{prop-new}, \Cref{coro-reverse} provides a polyhedral characterization
of $\val (\TN)$.
\end{rem}
\begin{rem}
In~\cite[Theorem~1.6.4]{Fallat&Johnson} it is shown that every double echelon
pattern is the Boolean pattern of some~$\TN$ matrix. The above result can
be recovered as a corollary of \Cref{eqrltn}. Indeed, if $P$ is a Boolean
matrix, we can define the tropical
matrix $A$ such that $A_{ij}=0$ if $P_{ij}=1$, and $A_{ij}=\minusinfty$
if $P_{ij}=0$. If $P$ has a double echelon pattern, then $A$
belongs to $\TN^\trop_2=\TN^\trop$. Then, choosing any $\puiseuxB\in \TP_{2,(n-1)^2}(\R)$,
we get that the matrix $\puiseuxA$ such that $\puiseuxA_{ij}=t^{A_{ij}}\puiseuxB_{ij}$ belongs to $\TN(\K)$. By substituting $t$ by a suitably large
real number, we end up with a matrix which is in $\TN(\R)$ and which
has the same pattern as the matrix $P$. 

Conversely, 
\Cref{TNDE} implies that if $\puiseuxA$ is a $\TN$ matrix without zero row or column, then, the pattern of $\puiseuxA$ is double echelon, 
recovering~\cite[Coro.~1.6.5]{Fallat&Johnson}. 
\end{rem}

\begin{rem}\label{rem-semialgebraic}
The relation between the set of double echelon matrices and the set of tropical totally nonnegative matrices can be understood in terms
of nonarchimedean amoebas of a semialgebraic
set, defined in \S\ref{subsec-nonarch}.
Consider the semi-algebraic set $\mathcal{S}\subset \R^{n\times m}$ consisting of totally
nonnegative matrices with no zero row or column, over a field $\mathscr{K}$.
Let us first take $\mathscr{K}=\R$, equipped with the trivial valuation, $\val$, which
sends any non-zero element to $0$, and $0$ to $-\infty$.
Then, the results of~\cite{Fallat&Johnson} which we just recalled
mean that $\val(\mathcal{S})$ is the set of 
matrices with a double echelon pattern. If we take $\mathscr{K}=\K$,
the field of Puiseux series, with the standard nonarchimedean valuation $\val$,
and if we consider $\mathcal{S}:=\TP(\K)$, \Cref{trop2} characterizes the amoeba
of $\val(\mathcal{S})$. 
In this way, classical results made in combinatorial matrix theory,
concerning Boolean patterns, appear to be related with tropical results:
they all concern nonarchimedean amoebas, albeit with different valuations.
\end{rem}

\begin{cor}\label{diagcor} $\val(\TP)=\val(\TN(\mathbb{K}^*))\enspace.$
\end{cor}
\begin{proof}
Trivially, $\val(\TP)\subset\val(\TN(\mathbb{K}^*))$.
By \Cref{coro-reverse}, we have~$\val(\TN(\mathbb{K}^*))\subset
\TN^\trop(\mathbb{R})$, and by the last statement in~\Cref{eqrltn},
for every matrix $A\in \TN^\trop(\mathbb{R})$, we can find 
$\puiseuxA\in \TP$ such that $\val \puiseuxA=A$,
showing that $\TN^\trop(\mathbb{R})\subset \val(\TP)$.
\end{proof}

We  conclude from 
Theorem~\ref{rltn} and Theorem~\ref{eqrltn} that the 
set~$\{A\in\TN^\trop:\per(A)\ne\zero\}$ coincides with the set of matrices 
generated by tropical Jacobi matrices.

The following theorem summarizes our results.
\begin{customthm}{C}\label{thC}
\label{diagramtp}
The different classes of matrices considered so far
satisfy the relations shown in \Cref{Table1,Table2}.
\end{customthm}
\begin{proof} 
This follows from Theorem~\ref{trop2}, Proposition~\ref{DD}, 
Theorem~\ref{SC}, Theorem~\ref{rltn},   Theorem~\ref{eqrltn}, \Cref{coro-reverse} and Corollary~\ref{diagcor}.
\end{proof}

\newcommand{\myline}{\rule{14cm}{0.4pt}}
\begin{table}[ht]
\myline\\[1mm]
$$\begin{array}{ccccccccccc}
 \TP^\trop             &  \subsetneqq    &   {\val(\TP) }                  &   =& \val(\TN(\mgroup{\mathbb{K}}))   &  \subsetneqq  &      \val(\TN)        &                       = &     \TN^\trop      \\
&&&&&&&&\\ 
   \parallel              &                        &                                      &      &                     {\parallel }                 &                       &                              &                        &       ||                    \\  
&&&&&&&&\\  
\TP^\trop_2          &                        &                                      &       &           \TN^\trop(\mathbb{R})      &                      &                               &                         & \TN^\trop_2          \\ 
\text{(strict Monge)}&&&&\text{(Monge)}&&&&\\ 
&&&&&&&&\\
\nshortparallel \cap&                       &                                      &       &              \parallel                           &                     &                               &                           &     \nshortparallel \cap             \\  
&&&&&&&&\\
      \NDD               &                       &                                      &       &          \text{Staircase matrices}     &                      &                              &                            &                \DD                \\        
&&&&(\text{upto diagonal scaling})&&&&\end{array}$$         
\caption{Images by the valuation of nonarchimedean totally nonnegative matrices compared with tropical totally nonnegative matrices}
\myline
\label{Table1}
\end{table}
\begin{table}[ht]
\myline\\[1mm]
$$\begin{array}{ccccccc}
\val(\TP)     &   \subsetneqq   &     \val(\GL\cap\TN)      &=&  \langle\text{tropical~Jacobi~matrices}\rangle    &   \subsetneqq   &     \val(\TN) \\
&&&&&&\\
      &&           \parallel  &                 &                            &                 &                   \\
&&&&&&\\
       &&        \{A\in\TN^\trop:\per(A)\ne\zero\}    &            
    &    
&                 &               \\
 &&&&&&
\end{array}$$            
\caption{The special case of square matrices}
\label{Table2}
\myline
\end{table}

\begin{rem}\label{rem-complem}
The results of \Cref{diagramtp} imply that
\begin{align}
\val(\GL\cap \TN(\K^*)) = 
\val(\GL\cap \TN)\cap \R^{n\times n} = 
\val(\TP)\label{e-finitestratum}
\end{align}
Indeed, we have $\TP\subset \GL\cap \TN(\K^*)$, and
so, $\val(\TP)\subset \val(\GL\cap \TN(\K^*))$. 
The inclusion 
$\val(\GL\cap \TN(\K^*)) \subset
\val(\GL\cap \TN)\cap \R^{n\times n}$
is trivial. Now, if 
$M\in \val(\GL\cap \TN)\cap \R^{n\times n}$,
we have
$M=\val \mathbf{M}$ with $\mathbf{M}\in \TN(\K^*)
= \val(\TP)$, showing that
$\val(\GL\cap \TN)\cap \R^{n\times n}\subset\val(\TP)$.
\end{rem}
\begin{rem}\label{rem-stratum}
Any subset $S$ of $\rmax^{n\times n}$ can be decomposed in {\em strata}, consisting of those matrices in the set $S$ that share the same pattern (i.e., whose
finite entries are in prescribed positions). 
It follows from~\Cref{rltn}
and~\eqref{e-finitestratum}
that
\[ 
\langle \text{tropical Jacobi matrices} \rangle  \cap \R^{n\times n}= 
\text{Monge matrices}
\]
providing a polyhedral characterization of the main stratum
of the semigroup $\langle \text{tropical Jacobi matrices} \rangle$, consisting
of matrices with finite entries. 
This semigroup has other strata which it may be interesting
to characterize.
\end{rem}

\begin{rem}\label{rk-jacob-diff-gaussian}
Recall that an (elementary) \textit{Gaussian matrix}
 differs from  the identity matrix by at most one entry, that is non-zero.
Analogously, an (elementary) \textit{tropical Gaussian matrix}
 differs from  the tropical identity matrix by at most one entry, that is finite. In particular, every (tropical) Jacobi matrix 
is a (tropical) Gaussian matrix.
We characterized in \Cref{diagramtp} the multiplicative semigroup generated by tropical Jacobi matrices. It is an interesting open question to characterize
the semigroup $\langle \text{Tropical Gaussian matrices}\rangle$.
The case of of $3\times 3$ matrices is solved in~\cite[Lemma~4.5]{FTM}, stating that a matrix $A\in\rmax^{3\times 3}$ is not factorizable  as
a product of an invertible matrix and of Gaussian matrices
if and only if there exists a
fixed point-free permutation $\sigma\in S_3$
satisfying
$$A_{i,\sigma(i)}<A_{i,k}A_{k,\sigma(i)},\ k\ne i,\sigma(i),\ \ \forall i.$$
\end{rem}

\section{Characteristic polynomial and eigenvalues of tropical totally nonnegative matrices}\label{cpev} 
Gantmacher and Krein established the following property: a totally nonnegative 
matrix has nonnegative eigenvalues, which are distinct and positive 
when the matrix is totally positive, see~\cite{Gantmacher&Krein}
and also~\cite[\S5.2]{Fallat&Johnson} for a more recent treatment.
The field $\mathbb{K}$ is real closed and
the theorem of Tarski,
which we recalled in \Cref{subsec-nonarch},
shows that
the first order theory of real closed fields is complete.
It follows that the property found
of Gantmacher and Krein remains valid
for matrices with entries in $\mathbb{K}$.


We shall see that the tropical eigenvalues of 
a $\TN^\trop$ matrix $A$ coincide with its diagonal entries, and that they coincide with 
the images by the valuation of the eigenvalues of an arbitrary  $\TN$ lift of $A$. 
However, the tropical eigenvalues are not necessarily distinct when $A\in\TP^\trop$.

We recall that the characteristic polynomial of~$\puiseuxA\in
\mathbb{K}^{n\times n}$ is defined by \begin{equation}\label{CPC}\puiseuxf_{\puiseuxA}
(\lambda)=\det(\lambda \puiseuxI-\puiseuxA)=\lambda^n+
\sum_{i=0}^{n-1}(-1)^{n-i}\puiseuxalpha_{n-i}\lambda^i,\end{equation}
where $\puiseuxI$ is the identity matrix
and $\puiseuxalpha_k$ is the sum of all the $k\times k$ 
principal minors of~$\puiseuxA$, with~$\puiseuxalpha_0:=1$.

Similarly, the tropical characteristic polynomial of~$A\in\rmax^{n\times n}$
 is defined as
\begin{equation}\label{CPT}f_A(x)=\per\big(x\odot I\oplus A\big)=x^{\odot n}\oplus
\bigoplus_{i=0}^{n-1}  
a_{n-i}\odot x^{\odot i},\end{equation} where~$I$ is the tropical identity matrix,~$a_k$ is the maximum between 
the weights of all~$k\times k$   principal tropical minors of~${A}$, and~$a_0:=\unit$.

In~\cite{MA} Butkovic discussed connections between max-algebraic problems and combinatorial optimization  problems. He observed that computing the coefficient $a_k$ of the tropical characteristic polynomial is equivalent to a modification of the optimal assignment problem, the \textit{optimal principal submatrix problem}, with
an unsettled complexity in general. The next result shows that in the special case of $\TN^{\trop}$ matrices, the $a_k$ are easy to compute.
\begin{pro}
Suppose that $A\in \TN^{\trop}$. Then, the coefficient $a_k$ of the tropical
characteristic polynomial of $A$ coincides with the product of the $k$ largest
diagonal entries of $A$.
\end{pro}
\begin{proof}
Let us ordering the diagonal entries of $A$ from the smallest to greatest, we denote  by~$A_{i_j,i_j}$ the~$j^{th}$ greatest diagonal entry of $A$. That is, $A_{i_1,i_1}
\geq A_{i_2,i_2}\geq \dots\geq A_{i_n,i_n},$  where $\{i_1,\dots,i_n\}=[n].$  
From~\Cref{DD},~$A\in \DD,$ and therefore,  for every $I\subset [n]$
of cardinality $k$, we have $\per A_{II} = \bigodot_{i\in I} A_{ii}$.
Since $a_k$ is the maximum of all the terms of this form, we deduce that
\begin{equation}\label{cc}a_k = A_{i_1i_1} \odot \dots \odot A_{i_ki_k} \enspace .\end{equation} 
\end{proof}


The tropical roots of a tropical polynomial function are generally defined as the points of which the maximum of the monomials appearing in this polynomial is achieved twice at least. In particular, a finite tropical root is a nondifferentiability point of this function. We define tropical eigenvalue by specializing this definition
to the characteristic polynomial. 
\begin{df}
We say that~$\eta\in \rmax$ is a {\em tropical eigenvalue} of~$A$ if 
$$\exists i\ne j\ \text{ s.t. }f_A(\eta)= a_{n-i}\odot \eta^{\odot i}\oplus a_{n-j}\odot \eta^{\odot j}.$$
 We say that $a_k$ is \textit{active} if there exists an eigenvalue $\eta$ such that
$$f_A(\eta)= a_{k}\odot \eta^{\odot n-k}.$$
The \textit{multiplicity} of a tropical eigenvalue~$\eta \neq \minusinfty$ is the difference between the right 
derivative and left derivative of~$f_A$ at this  nondifferentiability point. 
If $\minusinfty$ is a tropical eigenvalue, its multiplicity is defined as the limit of of the derivative of $f_A$ at  point $t$, as $t\to \minusinfty$.
\end{df}
Observe that counting multiplicities, an~$n\times n$ tropical matrix~$A$ has exactly~$n$ eigenvalues.

\begin{pro}\label{eede} If~$A\in\TN^{\trop},$ then $a_k$ is \textit{active} for all $k\in\{0\}\cup[n] ,$ 
and the tropical eigenvalues of $A$ (counting multiplicities)
are precisely its diagonal entries.\end{pro}
\begin{proof} 
It follows from \eqref{cc} that
\begin{align*}
f_A(x) &= x^{\odot n} \oplus  A_{i_1i_1}x^{\odot n-1} \oplus \dots \oplus A_{i_1i_1}\odot\dots\odot A_{i_ni_n}
\enspace .
\end{align*}
Since 
$A_{i_1,i_1} \geq \dots \geq A_{i_ni_n}$, the above polynomial
function can be rewritten as
\[
f_A(x) = (x\oplus A_{i_1i_1})\odot\dots \odot(x\oplus A_{i_ni_n}) \enspace.
\]
Therefore, the $k$th tropical eigenvalue of $A$ is given by $\eta_k=A_{i_ki_k}$.
We also deduce that
\[ f_A(A_{i_k,i_k})=a_k \odot A_{i_k,i_k}^{\odot n-k} \enspace, 
\]
showing that $a_k$ is active.
\end{proof}




Proposition~\ref{cc&ee} below is illustrated by the following example.

\begin{exa}\label{ev} The characteristic polynomial~$\puiseuxf_{\puiseuxA}(\lambda)
=\big(\lambda-(t^{2}+t)\big)\big(\lambda-(t^{2}-t)\big) $ of 
 $$\puiseuxA=\left(\begin{array}{cc}t^{2}&t\ \\\
t\ \ &t^{2}\end{array}\right)\in \TP,$$  has positive distinct roots.
However, considering its valuation $$\val(\puiseuxA)=A=\left(\begin{array}{cc}2&1\\
1&2\end{array}\right)\in \TP^{\trop},$$ we get the tropical characteristic 
polynomial~$f_{A}(x)=x^{\odot 2}\oplus 2\odot x\oplus 4,$ with a tropical eigenvalue~$2$ of
 multiplicity~$2$. Nevertheless,   
the tropical coefficients and tropical eigenvalues  are the corresponding images by the valuation of the  
coefficients and  eigenvalues of~$\puiseuxA$. 
This may fail for $\puiseuxA\in \TP$ with $\val(\puiseuxA)\notin 
\TP^\trop$, as seen by the following example.

The valuation of the coefficients of the characteristic polynomial~$\lambda^2-(t+2)\lambda+1$ of 
 $$\puiseuxA=\left(\begin{array}{cc}t+1&t\ \\\
1&1\end{array}\right)\in \TP,$$  do not coincide with the coefficients of the 
tropical characteristic polynomial~$x^{\odot 2}\oplus 1\odot x\oplus 1$
of $$\val(\puiseuxA)=\left(\begin{array}{cc}1&1\\
0&0\end{array}\right)\notin \TP^{\trop}.$$ 
\end{exa}


Denote the characteristic polynomial of~$A\in (\TP^{\trop})^{n\times n}$ 
as in~\eqref{CPT}, with the 
tropical eigenvalues~$\eta_1\geq\dots \geq\eta_n$. 
By Theorem~\ref{rltn2}, we know that   $A\in\TP^{\trop}$ if and only if 
 every lift $\puiseuxA$ of $A$ is in $\TP$.
We denote the characteristic polynomial of $\puiseuxA$ as in~\eqref{CPC},
and we denote by $\puiseuxrho_1>\dots >\puiseuxrho_n$
the eigenvalues of $\puiseuxA$. 

The next result characterizes the valuation of the coefficients
of the characteristic polynomial and of the eigenvalues
of a totally positive matrix $\puiseuxA$ with entries in $\K$.
\begin{pro}\label{cc&ee}
If $\puiseuxA\in \TP$ is such that  $A=\val (\puiseuxA)\in\TP^{\trop}$,
we have
\begin{eqnarray}
\val(\puiseuxalpha_k)=a_{k}&,&\forall k\in\{0\}\cup[n].\label{1}\\
\val(\puiseuxrho_i)=\eta_i\ &,&\forall i\in[n].\label{2}\end{eqnarray}
\end{pro}

\begin{proof}

Recall that applying the valuation to strict inequalities
over~$\mathbb{K}_{\geq 0}$ yields weak inequalities.

We prove~\eqref{1}. 
In every submatrix $M$ of the lift~$\puiseuxA$ of~$A$, the valuation of  product 
of diagonal entries of $M$ is the only one of maximal weight in $\val(\det(M))$.
 Since~$\puiseuxalpha_k$ 
is the sum  of~$k\times k$ principal minors of~$\puiseuxA$, we get 
that~$\val(\puiseuxalpha_0)=a_0$ and for every~$k\in[n]$
\begin{eqnarray*}
&\val(\puiseuxalpha_k)&
=\max_{\substack{I\subset[n]:\\|I|=k}}\val\Bigg(\sum_{\sigma\in S_{I}} \sign(\sigma)\prod_{i\in I}
\puiseuxA_{i,\sigma(i)}\Bigg)=
\max_{\substack{I\subset[n]:\\|I|=k}}
\left\{\val\left(\prod_{i\in I}\puiseuxA_{i,i}\right)\right\}\\
&&
=\max_{\substack{I\subset[n]:\\|I|=k}}
\left\{\sum_{i\in I}\val(\puiseuxA_{i,i})\right\}=
\bigoplus_{\tiny{\begin{array}{c}I\subset[n]:\\|I|=k\end{array}}}
\bigodot_{i\in I}A_{i,i}=\bigoplus_{\substack{I\subset[n]:\\|I|=k}}
\bigoplus_{\sigma\in S_{I}}\bigodot_{i\in I}A_{i,\sigma(i)}
=a_k.\\
&&\end{eqnarray*}

We now prove~\eqref{2}. By Puiseux Theorem, we know that the
image by the valuation of the roots of the polynomial 
$\puiseuxf_\puiseuxA(\lambda)$ are precisely the slopes of the 
associated Newton polytope, defined
as the upper boundary of the concave hull of the points $(i,\val (\puiseuxalpha_{n-i}))$, 
$i=0,\dots, n$. By Legendre-Fenchel duality, these are precisely the
tropical roots of
\begin{align} x\mapsto \max_{i\in[n]}\{ix+\val(\puiseuxalpha_i)\} \enspace,
\label{dual-puiseux}
\end{align}
counted with multiplicities.
In general~$\val(\puiseuxalpha_i)\leq a_i$, but   equality holds 
when~$A\in \TP^\trop$, 
as seen in the first item. Therefore, the tropical eigenvalues
of~$A$ coincides with the tropical roots of the polynomial 
function~\eqref{dual-puiseux}.
\end{proof}

\begin{rem}
In  general, it is known that
the sequence of valuations of the eigenvalues of a matrix~$\puiseuxA$
with entries in~$\K$ is weakly majorized by the
sequence of tropical eigenvalues of the matrix obtained
by applying the valuation to every entry of~$\puiseuxA$,
see Theorem~4.4 of~\cite{MPA16}. The previous result shows that
 equality holds when~$\puiseuxA\in\TP$. 
\end{rem}
\begin{rem}\label{rem-eigen}
We saw that the eigenvalues of a tropical totally nonnegative matrix
have a simple characterization. We may also consider
eigenvectors, i.e., solutions $u$ of $Au=\lambda u$
for some scalar $\lambda$. Such eigenvectors can be determined,
by exploiting the general combinatorial characterization
of tropical eigenvectors, in terms of shortest paths matrix
 (see for instance~\cite[\S~4.2]{butkovicbook}).
It may be interesting to investigate whether eigenvectors
show more structure in the tropical totally positive case.
\end{rem}

\section{Other aspects of tropical total positivity}\label{Other aspects}

\subsection{Tropical totally positive matrices and the tropical totally positive Grassmannian}\label{subsec-other}
We next relate the set of totally positive matrices with the totally
positive tropical Grassmannian considered by Postnikov, Speyer and Williams~\cite{POST,SW05}.

To any $n\times m$ matrix $\puiseuxA$ with entries in a field, we associate a vector of \textit{Pl\"ucker coordinates},
whose entries are the maximal minors.
We denote by~$\Delta(\puiseuxA)$ the vector of
Pl\"ucker coordinates of this matrix, assuming
without loss of generality that $n\leq m$.
So, $\Delta(\puiseuxA)$ is a vector of size $C_m^n$ whose entries are indexed
by the subsets $I\subset[m]$ such that $|I|=n$, ordered
lexicographically. We denote  the
~${I}$-coordinate of~$\Delta(\puiseuxA)$ by~$\Delta_I(\puiseuxA)
:= \det \puiseuxA_I$, where $\puiseuxA_I$ is the maximal submatrix
of $\puiseuxA$ of column set $I$. 
If $A$ is a $n\times m$ tropical matrix, the tropical Pl\"ucker coordinates
are defined analogously. We still denote them by $\Delta(A)$. In particular, we have $\Delta_I(A)=\per(A_I)$, where now $\per$ denotes the tropical permanent of the
maximal submatrix $A_I$. 

The  \textit{Grassmannian}~$\Gr_{k,n}$ on $\K$ is the space of~$k$-dimensional
subspaces of~$\K^n$. An element of~$\Gr_{k,n}$ can be represented by
an~$k\times n$ matrix of  full-rank, modulo  left multiplication by~$\GL_k$.
The map~$\puiseuxA \mapsto \Delta(\puiseuxA)$ yields
an \textit{embedding}~$\Gr_{k,n} \hookrightarrow 
\mathbb{P}^{\left(\substack{n\\k}\right)-1}(\K)$.
The \textit{totally  positive Grassmannian}~$\Gr_{k,n}^+\subset \Gr_{k,n}$,
studied in~\cite{POST,SW05},
is the
subset of~$k$-dimensional
subspaces that can be represented by  matrices~$\puiseuxA\in\K^{k\times n}$
with~$\Delta_I(\puiseuxA)>0$ for all subsets $I$ of $[n]$ with $k$ elements.
The \textit{totally nonnegative Grassmannian}~$\Gr_{k,n}^{\geq 0}\subset \Gr_{k,n}$ is defined in a similar way, requiring this time that 
$\Delta_I(\puiseuxA)\geq 0$ for all subsets $I$ of $[n]$ with $k$ elements.

There is a known correspondence between the matrices of 
$\puiseuxB\in\TP^{k\times(n-k)}$ and the elements of the 
totally positive Grassmannian $\Gr_{k,n}^+$. To see this, we first
associate with $\puiseuxB$ the matrix 
\begin{equation}\label{const}
\imath(\puiseuxB):=\left(\begin{array}{ccccccc}
1&\dots&0&0&         (-1)^{k-1}b_{k,1}&\dots&(-1)^{k-1}b_{k,n-k} \\
  \vdots & \ddots & \vdots&    \vdots&  \vdots & \ddots & \vdots\\
0&\dots&1&0  &       -b_{2,1} & \dots &- b_{2,n-k}\\
 0&\dots&0&1&        b_{1,1} & \dots & b_{1,n-k}
\end{array}\right).\end{equation}
In other words, 
$\imath(\puiseuxB)=(\mathcal{I}\ |\ \tilde{\puiseuxB})\in\K^{k\times n}\  $, 
where $\mathcal{I}$ is the $k\times k$ identity matrix, 
and the matrix $\tilde{\puiseuxB}$ is defined by
$\tilde{\puiseuxB}_t=(-1)^{k-t} \puiseuxB_{k-t+1}$, for all $t\in[k]$.

One can check that 
\begin{equation}\label{corres}
\det 
(\puiseuxB_{I,J})=\Delta_{([k]\setminus \tilde{I})\cup\tilde{J}}(\imath(\puiseuxB)),\end{equation} 
where $\puiseuxB_{I,J}$ denotes the $I\times I$ submatrix of $\puiseuxB$, 
$\tilde{I}=\{k-t+1:\ t\in I\}$, and $\tilde{J}=\{t+k:\ t\in J\}$. 
We consider the following map
\[ \stiefel: \puiseuxB\mapsto \plucker(\imath(\puiseuxB)) \enspace .
\]
Postnikov made the following observation in~\cite{POST},
up to a trivial modification (the entries of the matrix
$b$ are listed from bottom to top in~\eqref{const}).
\begin{pro}[See Prop.~3.10 of~\cite{POST}]\label{post}
The map $\stiefel$ sends bijectively $\TP^{k\times (n-k)}$
to the totally positive Grassmannian 
$\Gr_{k,n}^+$.
\end{pro}



We next compare the set of tropical totally positive matrices with
the \textit{tropical totally positive Grassmannian}~$\Trop^+(\Gr):=\val(\Gr_{n,k}^+)$, studied in~\cite{SW05}. Here, $\Trop^+(\Gr)$ is
thought of as the subset of the tropical projective space 
$\mathbb{P}^{C_n^k-1}(\rmax)$, 
obtained as the image by the valuation of $\Gr_{n,k}^+$, the latter being
identified to its image by the Pl\"ucker embedding.

We now associate 
to a matrix $B\in \R^{k\times (n-k)}$ the matrix
\begin{align}\label{const2}
\imath(B) :=
\left(\begin{array}{ccccccc} 0&\dots&\minusinfty&\minusinfty&         b_{k,1}&\dots&b_{k,n-k} \\
  \vdots & \ddots & \vdots&    \vdots&  \vdots & \ddots & \vdots\\
\minusinfty&\dots&0&\minusinfty  &       b_{2,1} & \dots & b_{2,n-k}\\
 \minusinfty&\dots&\minusinfty&0&        b_{1,1} & \dots & b_{1,n-k}
\end{array}\right)
\end{align}
and define the map
$\stiefel: \R^{k\times(n-k)}\to \R^{C_n^k}$,  such that $\stiefel(B)$ is the vector with entries
\[
\stiefel_I(B) = \per (\imath(B))_I
\]
for any subset $I$ of $[n]$ with $k$ elements. 
The next result shows that in the tropical setting, the map $\stiefel$ is still injective. We shall see in \Cref{tnnegr} that it is no longer surjective.

\begin{pro} \label{tropcorres} \ 
The map $\stiefel$
sends linearly $(\TN^\trop(\R))^{k\times(n-k)}$ to a closed
polyhedral subset of $\Trop^+(\Gr_{n,k})$. 



\end{pro}

\begin{proof}  
The construction of \Cref{eqrltn} implies that any matrix $B\in \TN^{k\times (n-k)}(\R)$ has a lift
$\puiseuxB\in\TP$, meaning that $B=\val (\puiseuxB)$. 
Moreover, by~\eqref{e-identnew}, this lift can be chosen such that
\begin{align} \per B_{I,J} = \val \det \puiseuxB_{I,J}
\label{e-liftcommute}
\end{align}
for all subsets $I\subset [k],J\subset [n-k]$ that have the same
cardinality.
Then, by \Cref{post}, $\stiefel(\puiseuxB)\in \Gr^+_{k,n}$.
Note also that
\begin{equation}\label{corres2}
\per
({B}_{I,J})=\Delta_{([k]\setminus \tilde{I})\cup\tilde{J}}({\imath(B)}).\end{equation} 
We deduce from~\eqref{e-liftcommute} and~\eqref{corres2} that
$\stiefel (B) = \val \stiefel(\imath(\puiseuxB))\in \Trop^+\Gr_{k,n}$. 
Moreover, if $B,B '$ are distinct elements of $\TN^{k\times (n-k)}(\R)$, then,
it follows from~\eqref{corres2} that $\stiefel(B)$ and $\stiefel(B')$ have at least one distinct coordinate. 
Observe that $\stiefel(B)$ and $\stiefel(B')$ have also one
identical coordinate, namely, the one corresponding to the choice $I=J=\varnothing$ in~\eqref{corres2}. It follows that the vectors $\stiefel(B)$ and $\stiefel(B')$ are non proportional in the tropical sense, and so, they define distinct
elements of $\Trop^+(\Gr_{k,n})$. We already observed that for a
square matrix $C$ in $\TN(\R)$, the maximum in evaluating $\per C$ is always
achieved on the diagonal. It follows that the map $\stiefel$ {\em restricted}
to $\TN(\R)$ is linear (it is only piecewise linear in $\R^{k\times (n-k)}$).
Hence, this map sends the set $\TN(\R)$, which is polyhedral,
to a closed polyhedron contained in $\Trop^+(\Gr_{k,n})$.
\end{proof}
\begin{exa}\label{illustrative}
Consider the following instance of the construction~\eqref{const2}:
\[ B = 
\left(\begin{array}{cc}
  0&-2\\
0&-1\\
0&0 
\end{array}\right)\in \TP^\trop,
\qquad 
A:=\imath(B)=\left(\begin{array}{ccccc}
0&\minusinfty&\minusinfty&  0&0\\
\minusinfty&0&\minusinfty&0&-1\\
\minusinfty&\minusinfty&0&0&{-2} 
\end{array}\right)\enspace,
\]
The image of $B$ by the 
map $\stiefel$ is the point
of the tropical projective space with coordinates
\[
\Delta(A)=(0: 0: -2: 0: -1: -1 : 0 : 0 : 0 : 0 )
\]
\Cref{tropcorres} shows that this point belongs
to $\Trop^+\Gr_{3,5}$. This can be checked by elementary means
as follows. Since $B\in \TP^\trop$, we do not need to consider the special
lift of \Cref{eqrltn}, any lift of $B$, and in particular
the trivial lift $\puiseuxB=(t^{B_{ij}})$, belongs to $\TP(\K)$. Then, the matrix
$\puiseuxB$ is sent to the following matrix by the construction~\eqref{const}
$$
\imath(\puiseuxB)=\left(\begin{array}{ccccc}
1&0&0& \  1&1\\
0&1&0&-1&-t^{-1}\\
0&0&1&\ 1& t^{-2}
\end{array}\right) \enspace .
$$
One can check that
\[
\Delta(\imath(\puiseuxB))= (1: 1: t^{-2}: 1: t^{-1}: t^{-1}-t^{-2}: 1: 1: 1-t^{-2}: 1-t^{-1})\in \Gr^+_{3,5}
\]
and so, $\val \Delta(\imath(\puiseuxB))=\Delta(A)\in \Trop^+\Gr_{3,5}$.

\end{exa}

\begin{exa}
\label{tnnegr} 
We next give an elementary example showing that $\stiefel((\TN^\trop)^{2\times 2})$ is a strict subset of $\Trop^+\Gr_{2,4}$. Consider the matrix 
$$D=\left(\begin{array}{cccc}0&-1&-2&-3\\
0&0&0&0
\end{array}\right),$$
and the trivial lift $\mathbf{D}:=(t^{D_{ij}})$. It can be checked
that
\[
\Delta(\mathbf{D})=(1-t^{-1}:1-t^{-2}:1-t^{-3}:t^{-1}-t^{-2}:t^{-1}-t^{-3}:t^{-2}-t^{-3})\in \Gr^+_{2,4}
\]
and so, $\val(\Delta(\mathbf{D}))=\Delta(D)=(0:0:0:-1:-1:-2)\in \Trop^+\Gr_{2,4}$.
However, this element of $\Trop^+\Gr_{2,4}$ does not belong to the 
image of $\TN^\trop(\R)$ by the map $\phi$. Indeed, assume by contradiction
that $\Delta(D)=\stiefel(C)$ for some matrix 
$C$ of size $2\times 2$. 
Using \eqref{corres2}, we get that 
$0=\Delta_{\{1,3\}}(D)=C_{1,1}$, 
$0=\Delta_{\{1,4\}}(D) =  C_{1,2}$, 
$-1=\Delta_{\{2,3\}}(D) =  C_{2,1}$, 
$-1=\Delta_{\{2,4\}}(D)= C_{22}$,
and so
\[ C=\left(\begin{array}{cc}
0&0\\
-1&-1 
\end{array}\right) \enspace.
\]
However, using again~\eqref{corres2}, we see
that $-2=\Delta_{\{3,4\}}(D)=\per C_{\{1,2\},\{1,2\}}=-1$,
a contradiction.
\end{exa}

\begin{rem}\label{rk-fink-rincon}
The question of characterizing the image of the map $\stiefel$
has already been studied without considerations
of total positivity. Fink and Rinc\'on 
called ``tropical Stiefel map''~\cite{rinconfink}
the map $A\mapsto \Delta(A)$ sending a tropical matrix
to the vector of its maximal tropical minors. 
The image $\Delta(\R^{k\times n})$ is called the Stiefel
image.  Rinc\'on and Fink observed that
\begin{align}
\phi(\R^{k\times (n-k)})\subset \Delta(\R^{k\times n})
\subset \val (\Gr_{n,k}(\K)) \enspace .
\label{e-compare}
\end{align}
They showed that the second inclusion
is strict for $(k,n)=(2,6)$.  
This also follows from a result 
of Herrmann, Joswig and Speyer, see 
Corollary~14 and Example~10 in~\cite{joswigspeyer}.
Corollary 3.8 of~\cite{rinconfink} also allows
one to represent $\Delta(\R^{k\times n})$ as the
union of the images of a family of maps, including
the map $\phi$ as a special case. Hence, this result
implies that the first inclusion in~\eqref{e-compare}
is also strict. 
We leave it for further investigation to look
for refinements of these results in the case
of the tropical totally positive Grassmannian.
\end{rem}
\begin{rem}\label{rem-positroid}
A related issue is to investigate which part of the
totally {\em nonnegative} tropical Grassmanian
is given by the image $\phi(\TN^{\trop})$, i.e.,
by Pl\"ucker coordinates of the form $\phi(A)$ where now the matrix $A\in \TN^\trop$ (a tropical totally nonnegative matrix with possibly $-\infty$ entries). 
It may be interesting in particular to investigate relations
with Postnikov's positive Grassman cells.
Recall that Postnikov introduced in~\cite[Section~3]{POST} a decomposition of the totally nonnegative Grassmannian in terms of ``positive Grassman cells''. 
Such a cell is associated to a matroid $\mathcal{M}$; it consists
of those elements of $\Gr^+_{k,n}$ represented by a matrix $\puiseuxA$,
such that $\Delta_I(\puiseuxA)$ is nonzero if and only $I$ is a basis
of the matroid $\mathcal{M}$. The special matroids
corresponding to nonempty cells are called positroids. 
\end{rem}


\begin{rem}
Suppose~$A\in(\TN^\trop)^{n\times m}$ is in double echelon form,
with $n\leq m$, and that it has
sign-nonsingular tropical Pl\"ucker coordinates. Then, there 
exists~$B\in(\TP^\trop)^{n\times m}$ s.t.~$A$ and~$B$ represent the 
same element in~$\Trop^+\Gr_{n,m}$.

We construct the matrix $B$ as follows. First, we set
$B_{i,j}=A_{i,j}$, for all $i\leq j\leq m-n+i$ with $i\in[n]$ (in other words,
the diagonal entries of all the maximal submatrices of $A$ are unchanged). 
Since $A\in \TN^\trop$, the identity permutation in every maximal minor of~$A$ is of maximal weight. We deduce that every~$2\times 2$ 
solid submatrix~$A_{\{i,i+1\},\{j,j+1\}}=B_{\{i,i+1\},\{j,j+1\}}$ such 
that~$i+1\leq j\leq m-n+i-1,$ must be sign-nonsingular. Indeed, the maximal submatrix in~$A$ containing a solid $2\times 2$ sign singular matrix
must be sign-singular (there must be at least two permutations of maximal
weight in the latter matrix, namely the identity, and a transposition).
Then, for every~$i\in[n]$ we construct the~$B_{i,j}\ne \zero$ successively
for~$j=i-1,\dots,2,1$, by requiring that~$B_{\{i-1,i\},\{j,j+1\}}$   
satisfies the strict Monge property, and then successively
for~$j= m-n+i+1,\dots,m-1,m$, by requiring that~$B_{\{i,i+1\},\{j-1,j\}}$   
satisfies the strict Monge property.
\end{rem}

\subsection{Tropical totally nonnegative matrices and planar networks}
\label{stpn}
The combinatorial properties of minors of the weight matrix associated with a planar
 network are well known~\cite{F&Z}.
These were studied by Karlin and McGregor back in the 50's (see~\cite{KMc}).
Some applications were given
by Gessel and Viennot in~\cite{Gessel&Viennot,Gessel&Viennot2}. 
In this context, totally nonnegative matrices arise
as weight matrices of planar networks. We next show
that, as an immediate consequence of the previous
result, the same is true in the tropical setting.
To this end, we first recall or state some basic definitions.


A  graph is called \textit{planar} if it can be drawn on a plane so that its edges have only endpoint-intersections. 
A \textit{planar network} is a weighted directed 
planar graph, with no cycles. 
Throughout, we assume a network has~$n$ sources and targets, numbered bottom
 to top, with edges  directed left to right assigned with real weights.

Let~$G$ be a planar network with weights in~$\K$.
 The \textit{weight of a  path} between nodes~$i,j\in[n]$ in~$G$ (if exists)
 is the product of  weights of its edges. 
The \textit{weight matrix}  of~$G$ is an~$n\times n$ matrix 
having the sum of weights of all paths from~$i$ to~$j$ for its ~$i, j$-entry, and~$0$ 
if such a path does not exist. 

Suppose now that the edges of $G$ are weighted by elements of $\rmax$. Then,  the \textit{tropical weight matrix} 
of~$G$ has the paths of maximal weight
 from~$i$ to~$j$ for its~$i,j$-entry, and~$\minusinfty$ 
if such a path does not exist, where the  weight of a path becomes  the \textit{sum}
 of weights of its edges. 

\ 

\


\begin{exa} 
For all $\alpha\leq 6$, the  planar network

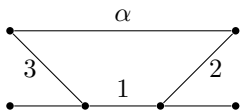
\begin{figure}[htbp]\begin{tikzpicture}[main_node/.style={circle,fill=black,minimum 
size=0.05em,inner sep=1pt]}]
 \node[main_node] (1) at (0,0) {};
    \node[main_node] (2) at (-1,1)  {};
    \node[main_node] (3) at (2,1) {};
    \node[main_node] (4) at (1,0) {};
    \node[main_node] (5) at (-1,0) {};
    \node[main_node] (6) at (2,0) {};
 \draw[main_node]  (1) edge node{$3\ \ \ \ $} (2);   
 \draw[main_node]  (2) edge node{$\begin{array}{c}\alpha\\ {}\end{array}$} (3);  
 \draw[main_node]  (3) edge node{$\ \ \ \ 2$} (4);    
 \draw[main_node]  (5) edge node{ } (1);    
 \draw[main_node]  (6) edge node{ } (4);    
 \draw[main_node]  (1) edge node{$\begin{array}{c}1\\ {}\end{array}$} (4);
\end{tikzpicture}
\end{figure}

\noindent 
corresponds to the tropical weight matrix 
\[ A=\left(\begin{array}{cc}1&3\\4&6\end{array}\right)\enspace .\]
\end{exa}

\begin{cor}\label{PNTP}
The tropical weight matrix of every planar network is tropical totally nonnegative, and every tropical totally nonnegative matrix is the tropical weight
matrix of some planar network.

\end{cor}
\begin{proof}
The weight matrix $\bm A$ of a planar network whose edges $e$ are weighted by nonnegative elements $\bm w_e \in \K$ is totally nonnegative
(see e.g.~ \cite[Coro 2]{F&Z}).
The valuation sends this weight matrix
to the tropical weight matrix $A$ arising by weighting the edge $e$ with $\val\bm w_e$. By Theorem~\ref{coro-reverse}, $A=\val \bm A$ is totally nonnegative.

Conversely, \Cref{rltn} shows that a finite tropical TN matrix $A$ can be
factored as a product of tropical Jacobi matrices. A classical result
allows one to identify a product of Jacobi matrices to the weight matrix
of a planar network, see the discussion before Theorem~13 of~\cite{F&Z}.
The same arguments works in the tropical setting, and we conclude
that $A$ is the tropical weight matrix of a planar network.

\end{proof}







\bibliographystyle{alpha}

\bibliography{tropical}

\end{document}